\documentclass[11pt,reqno]{amsart}
\usepackage{tikz}
\usepackage{subfig}
\usepackage{epstopdf}
\usepackage{epsfig}
\usepackage{math}
\usepackage{rotating}
\graphicspath{{../figures/},{"../../Code/hierarchical models/FiguresAdamNew/"},{"../../Code/hierarchical models/Type0/"}}
\usepackage{tikz}
\usetikzlibrary{shapes,arrows}
\tikzstyle{decision} = [diamond, draw, fill=blue!20, 
    text width=4.5em, text badly centered, node distance=3cm, inner sep=0pt]
\tikzstyle{block} = [rectangle, draw, fill=blue!20, 
    text width=5em, text centered, rounded corners, minimum height=4em]
\tikzstyle{line} = [draw, -latex']
\tikzstyle{cloud} = [draw, ellipse,fill=red!20, node distance=3cm,
    minimum height=2em]

\usetikzlibrary{positioning}
\tikzset{main node/.style={circle,fill=blue!20,draw,minimum size=1cm,inner sep=0pt},  }

\usepackage{soul}
\usepackage{color}
\newcommand{\wc}[1]{{\color{blue}{#1}}}
\usepackage{hyperref}
\usepackage{graphicx} \usepackage{placeins}

\begin{document}
\title[Wasserstein information matrix]{Wasserstein information matrix}
\author[Li]{Wuchen Li}
\author[Zhao]{Jiaxi Zhao}
\email{wcli@math.ucla.edu,zjx98math@gmail.com}

\newcommand{\vr}{\overrightarrow}
\newcommand{\wt}{\widetilde}
\newcommand{\dd}{\mathcal{\dagger}}
\newcommand{\ts}{\mathsf{T}}

\maketitle

\begin{abstract}
We study information matrices for statistical models by the $L^2$-Wasserstein metric. We call them Wasserstein information matrices (WIMs), which are analogs of classical Fisher information matrices. We introduce Wasserstein score functions and study covariance operators in statistical models. Using them, we establish Wasserstein-Cramer-Rao bounds for estimations and explore their comparisons with classical results. We next consider the asymptotic behaviors and efficiency of estimators. We derive the online asymptotic efficiency for Wasserstein natural gradient. Besides, we study a Poincar{\'e} efficiency for Wasserstein natural gradient of maximal likelihood estimation. Several analytical examples of WIMs are presented, including location-scale families, independent families, and rectified linear unit (ReLU) generative models.
\end{abstract}

\section{Introduction}
Fisher information matrix plays essential roles in statistics, physics, and differential geometry with applications in machine learning \cite{Amari1990_differentialgeometrical,NG,ay2017information,casella2002statistical,EIT}. In statistics, it is a fundamental quantity for the estimation theory, including both design and analysis of estimators. In particular, the maximal likelihood principle is a well-known example. It connects the Fisher information matrix to another concept, named score functions. They frequently arise in statistical efficiency and sufficiency problems, especially for Cramer-Rao bound and Fisher-efficiency.  

Fisher information matrix is also named Fisher-Rao metric in information geometry \cite{IG}. It uses the Fisher information matrix to study divergence functions and their invariance properties \cite{IG}. Furthermore, the Fisher information matrix is also useful for statistical learning problems. In particular, the natural gradient method \cite{NG} rectifies the gradient direction by the Fisher information matrix. It is shown that the Fisher natural gradient method is asymptotically online Fisher-efficient.    
\par
On the other hand, optimal transport introduces the other metric in probability space \cite{villani2003topics,villani2008optimal}, often named Wasserstein metric \cite{Lafferty,otto2001}. Different from information geometry, it encodes the geometry of sample space into the definition of metric in probability space.
Nowadays, it is known that the Wasserstein metric intrinsically connects the Kullback-Leibler (KL) divergence with Fisher information functional \cite{otto2001}, known as de Bruijn identities \cite{zozor2015debruijn}.
Many concentration inequalities such as log-Sobolev inequalities and Poincar{\'e} inequalities arise naturally \cite{otto2000generalization}. 

Despite various studies of optimal transport in full probability space, not much is known in parametric statistical models, which play crucial roles in parametric statistics. Fundamental questions arise: {\em Is there a statistical theory based on optimal transport? Compared to Fisher information matrices and Fisher statistics, what are counterparts of information matrices, score functions, Cramer-Rao bounds, and online efficiencies of natural gradient methods in Wasserstein statistics? Moreover, can this theory provide statistical tools for machine learning models, especially for generative models?}

In this paper, following key ideas in \cite{LiG}, we positively answer the above questions by introducing a Wasserstein information matrix (WIM). We derive the WIM by pulling back the Wasserstein metric from full probability space to finite-dimensional parametric statistical models \cite{LM,LiMontufar2018_riccia}. We show that the WIM defines Wasserstein score functions with a Wasserstein covariance operator of estimators. Based on them, a Wasserstein-Cramer-Rao bound is derived. Furthermore, combining WIM with Wasserstein score functions, we recover an asymptotic efficiency property of the online Wasserstein natural gradient methods.

Meanwhile, by comparing both Wasserstein and Fisher information matrices, we naturally prove several concentration inequalities such as log-Sobolev inequalities and Poincar\'e inequalities within statistical models. Extending the study in full probability space, we further decompose a Hessian term and study the Ricci-Information-Wasserstein (RIW) criterion for log-Sobolev inequalities and Poincar\'e inequalities in statistical models. Here we provide several examples in analytic probability families. Those functional inequalities turn out to be essential in a new efficiency property named Poincar{\'e} efficiency. This is concerned with dynamics where the Wasserstein natural gradient works on Fisher score functions (related to maximal likelihood estimators). We prove convergence rate analysis for these dynamics. Several numerical experiments are provided to confirm our conclusions.

Lastly, we demonstrate that the WIM provides a clear statistical theory for complicated models coming from machine learning approaches, especially implicit generative models. For example, we carefully study a one-dimensional probability family generated by push-forward maps based on the ReLU function. We demonstrate that the WIM still exists in this family while the classical Fisher information matrix does not exist. In other words, it is suitable to introduce a statistical theory based on WIMs. It can be a theoretical background for machine learning implicit models. 
\par

In literature, there have been lots of works attempting to use tools from optimal transport and information geometry to study statistical problems. \cite{2017arXiv170105146B} designs new estimators for parametric inference using Wasserstein distance. This idea is utilized in approximating Bayesian computation. The authors apply Wasserstein distance to measure the similarity between synthetic and observed data sets.
Compared to them, we focus on the study of estimation and efficiency of WIMs. We expect it could have potential properties in Wasserstein estimators. In \cite{briol2019statistical}, they design a generalized information matrix based on a maximum mean discrepancy. Compared to them, we majorly focus on information matrices generated by the Wasserstein metric and study related statistical properties. Most closely, \cite{petersen2019wasserstein} defines a Wasserstein covariance by applying a closed-form formula for one-dimensional Wasserstein metric. This is a canonical definition. Our approach further extends this idea into general parametric models. We start by introducing the WIM in parametric models. Using it, we define Wasserstein score functions as well as the Wasserstein covariance operator. We further establish the Wasserstein-Cramer-Rao bound and associated statistical efficiency properties. Also, \cite{Wong2018_logarithmic} defines several new divergence functions by combining knowledge from both optimal transport and information geometry. Here we focus on statistical properties of WIMs in statistical models. Furthermore, Wasserstein natural gradient method has been widely studied in optimization techniques with machine learning applications \cite{ArbelGrettonLiMontufar2019_kernelized,chen2018natural,mallasto2019formalization,li2019affine, lin2019wasserstein}. 
Here we focus on statistical theory and study its associated online efficiency. Compared with classical online Fisher-efficiency results in \cite{NG, Ollivier2018_online}, our results can deal with general information matrices. In particular, for WIMs, we discover a new efficiency property named Poincar\'{e} efficiency. It relies on a comparison between Wasserstein and Fisher information matrices, demonstrate its connection with Poincar\'{e} inequalities.

The paper is organized as follows. In section~\ref{sec2}, we establish the definition of the WIM. We present it analytically for several well-known probability families. We provide an explicit example of WIMs for ReLU generative models. Under this model, we show that the WIM exists while the Fisher information matrix does not exist. In section~\ref{sec-WCR}, with the introduction of the Wasserstein covariance, the Wasserstein-Cramer-Rao inequality is established. In section~\ref{Efficiency}, we introduce and discuss both Wasserstein efficiency and Poincar{\'e} efficiency.

\begin{table}[h]
\setlength{\belowcaptionskip}{-0.cm}
\setlength{\floatsep}{0cm}
\centering
    \resizebox{\textwidth}{0.20\textwidth}{
\begin{tabular}{|l|c|c|}
\hline 
Probability Family & Wasserstein information matrix & Fisher information matrix \\
\hline
$
\begin{aligned}
    & \text{Uniform:}        \\
    & p(x;a,b) = \frac{1}{b - a}\mathbf{1}_{(a,b)}(x)
\end{aligned} $
 & $G_W(a,b) = \frac{1}{3} \begin{pmatrix}
1 & \frac{1}{2}\\
\frac{1}{2} &  1
\end{pmatrix}$ & $G_F(a,b)$ not well-defined 
\\
\hline
$
\begin{aligned}
    & \text{Gaussian:}        \\
    & p(x;\mu,\sigma)=\frac{e^{-\frac{1}{2\sigma^2}(x-\mu)^2}}{\sqrt{2\pi} \sigma}
\end{aligned}$ & $G_W(\mu,\sigma)=\begin{pmatrix}
1& 0 \\
0 &  1
\end{pmatrix}$ &  $G_F(\mu,\sigma) = \begin{pmatrix}
			\frac{1}{\sigma^2} & 0		\\
			0 & \frac{2}{\sigma^2}
		\end{pmatrix}$    
\\
\hline
$
\begin{aligned}
    & \text{Exponential:}        \\
    & p(x;m,\lambda) = \lambda e^{-\lambda\lp x - m \rpz}
\end{aligned}$
& $G_W(m,\lambda)=\begin{pmatrix}
    1 & \frac{1}{\lambda^2}     \\
    \frac{1}{\lambda^2} & \frac{2}{\lambda^4}
\end{pmatrix}$ & $G_F(m,\lambda)$ not well-defined 
\\    
\hline
$
\begin{aligned}
    & \text{Laplacian:}        \\
    & p(x;m,\lambda) = \frac{\lambda}{2}e^{-\lambda|x-m|}
\end{aligned}
 $ & $G_W(m,\lambda)=\begin{pmatrix}
    1 & 0     \\
    0 & \frac{2}{\lambda^4}
\end{pmatrix}$ & 
$G_F(m,\lambda) = \begin{pmatrix}
			\lambda^2 & 0		\\
			0 & \frac{1}{\lambda^2}
		\end{pmatrix}$ 
\\    
\hline
$\begin{aligned}
    & \text{Location-scale:}        \\
    & p(x;m,\lambda) = \frac{1}{\lambda}p(\frac{x - p}{\lambda})
\end{aligned}$ & $G_W(\lambda,m)=\begin{pmatrix}
    \frac{\mathbb{E}_{\lambda,m}x^2-2m\mathbb{E}_{\lambda,m}x+m^2}{\lambda^2} & 0     \\
    0 & 1
\end{pmatrix}$ & 
$G_F(\lambda,m) = \begin{pmatrix}
			\frac{1}{\lambda^2} \lp 1 + \int_{\RR} \lp \frac{\lp x - m \rpz^2 p'^2}{\lambda^2 p} + \frac{\lp x - m \rpz p'}{\lambda}\rpz dx  \rpz & \int_{\mathbb{R}}  \frac{(x - m)p'^2}{\lambda^3p} dx		\\
			\int_{\mathbb{R}}  \frac{(x - m)p'^2}{\lambda^3p} dx & \frac{1}{\lambda^2}\int_{\mathbb{R}} \frac{p'^2}{p} dx
		\end{pmatrix}$ 
		\\    
\hline
$
\begin{aligned}
    & \text{Independent:}        \\
    & p(x,y;\theta) = p_1(x;\theta)p_2(y;\theta)
\end{aligned}
 $ & $G_W(\theta) = G_W^1(\theta) + G_W^2(\theta)$ &
$G_F(\theta) = G_F^1(\theta) + G_F^2(\theta)$    \\    
\hline
$
\begin{aligned}
    \text{ReLU} &\text{ push-forward:}        \\
    & p(x;\theta) = f_{\theta*}p(x), \\ f_{\theta} \ \theta&\text{-parameterized ReLUs, Ex. }\ref{ReLU-ex}.
\end{aligned}
 $ & $\begin{aligned}
        G_W \lp \theta \rpz & = F\lp \theta \rpz, \\
        F \text{ cdf of } p(x) & , \ F\lp y \rpz = \int_{-\infty}^y p(x)dx.
        \end{aligned}$ &
$G_F(\theta)$ not well-defined   \\    
\hline
\end{tabular}}
\label{table1}
\caption{In this table, we present Wasserstein, Fisher information matrices for probability families.}
\end{table}
\begin{table}[h]
\setlength{\belowcaptionskip}{-0.cm}
\setlength{\floatsep}{0cm}
\vspace{-0.0cm}
\centering
\resizebox{\textwidth}{0.12\textwidth}{
\begin{tabular}{|l|c|c|c|}
\hline 
 Family & Entropy functional & Fisher-information functional & Log-Sobolev inequality(LSI($\alpha$))\\
\hline
Gaussian & $\begin{aligned}
        \wtd H(p_{\mu,\sigma}) = & \ - \half \log 2\pi - \log \sigma - \half,         \\
        \wtd H(p_{\mu,\sigma}|p_{\mu_*, \sigma_*}) = & \  - \log \sigma + \log \sigma_* - \half \\
        & \ + \frac{\sigma^2 + \lp \mu - \mu_* \rpz^2}{2\sigma_*^2}.
    \end{aligned}$ & $\begin{aligned}
        \wtd I(p_{\mu,\sigma}) & = \frac{1}{\sigma^2},          \\
        \wtd I(p_{\mu,\sigma}|p_{\mu_*,\sigma_*}) & = \frac{\lp \mu - \mu_* \rpz^2}{4\sigma_*^4} + \lp - \frac{1}{\sigma} + \frac{\sigma}{\sigma_*^2} \rpz^2.
    \end{aligned}$&
    $\begin{aligned}
        \wtd H(p_{\mu, \sigma}|p_{\mu_*, \sigma_*}) & < \frac{1}{2\alpha}\wtd I(p_{\mu, \sigma}|p_{\mu_*, \sigma_*}),\\
        \mu, \sigma & > 0.
    \end{aligned}        $
	\\    
\hline
Laplacian & $\begin{aligned}
        \wtd H(p_{m,\lambda}) = & \ - 1 + \log \lambda - \log 2,           \\
        \wtd H(p_{m,\lambda}|p_{m_*, \lambda_*}) = & \  - 1 + \log \lambda - \log \lambda_* \\
        & \ + \lambda_*\lv m - m_* \rv + \frac{\lambda_* e^{ - \lambda\lv m - m_* \rv}}{\lambda}.    \end{aligned}$ &$\begin{aligned}
        & \wtd I(p_{\lambda,m}) = \  \frac{\lambda^2}{2},                \\
        & \wtd I(p_{m,\lambda}|p_{m_*,\lambda_*}) = 
            \lambda_*^2\lp 1 - e^{ - \lambda\lv m - m_* \rv}\rpz^2 \\
            & + \frac{\lp \lp \lambda\lv m - m_* \rv + 1 \rpz\lambda_*e^{ - \lambda\lv m - m_* \rv} - \lambda \rpz^2}{2}.
    \end{aligned}$ &     $
    \begin{aligned}
        \wtd H(p_{\lambda, m}|p_{\lambda_*, m_*}) < & \ \frac{1}{2\alpha}\wtd I(p_{\lambda, m}|p_{\lambda_*, m_*}),      \\
      m \in & \ \RR, \lambda > 0.
    \end{aligned}
        $
	\\
\hline
\end{tabular}}
\label{table2}
\caption{In this table, we continue to list the entropy functional, the Fisher information functionals, log-Sobolev inequalities for probability families.} 
\end{table}

\section{Wasserstein information matrix and score functions}\label{sec2}
In this section, we present Wasserstein information matrices (WIMs) and score functions. Several analytical studies are presented.

Given a sample space $\mathcal{X}\subset \mathbb{R}^n$, let $\mathcal{P}(\mathcal{X})$ denote the space of probability distributions over $\mathcal{X}$. Given a metric tensor $g$ on $\mathcal{P}(\mathcal{X})$, we call $(\mathcal{P}(\mathcal{X}),g)$ density manifold. Consider a parameter space $\Theta\subset \mathbb{R}^d$ and a parameterization function
\begin{equation*}
    p \colon \Theta \rightarrow \cp \lp\cx \rpz, \quad \theta \mapsto p_{\theta}
\end{equation*}
which can also be viewed as
$
p \colon \mathcal{X}\times \Theta\rightarrow \mathbb{R}. 
$
Here $\Theta$ is named a statistical model. Denote $\langle f, h\rangle=\int_{\mathcal{X}} f(x) h(x) dx$ for the $L^2(\mathcal{X})$ inner product, where $dx$ refers to the Lebesgue measure on $\mathcal{X}$. And we denote by $\lp v,w \rpz= v \cdot w$ the (pointwise) Euclidean inner product of two vectors.
\subsection{Information matrix}
We first review metric tensors on parameter space and connect them with information matrices.   
\begin{definition}[Statistical information matrix]
	Consider the density manifold $(\mathcal{P}(\mcX), g)$ with a metric tensor $g$, and a smoothly parametrized statistical model $ p_\theta$ with parameter $\theta\in\Theta \subset \RR^d$. 
	Then the pull-back metric $G \in \mbR^{d \times d}$ of $g$ onto this parameter space $\Theta$ is given by
	\begin{equation*}
	G(\theta)=\Big\langle \nabla_\theta p_\theta, g( p_\theta) \nabla_\theta p_\theta\Big\rangle.
	\end{equation*}
Denote $G(\theta)=(G(\theta)_{ij})_{1\leq i,j\leq d}$, then 
\begin{equation*}
G(\theta)_{ij}=\int_{\mathcal{X}}{\frac{\partial}{\partial \theta_i}p(x;\theta)\Big(g(p_\theta)\frac{\partial}{\partial \theta_j}p\Big)(x;\theta)}dx.
\end{equation*}
Here we name $g$ statistical metric, and call $G$ statistical information matrix.
\end{definition}
 In geometry, using this metric tensor $g$ to rise(resp. lower) indices, there exists a canonical isomorphism from tangent(resp. cotangent) space to cotangent(resp. tangent) space, namely:
    \begin{equation*}
        \begin{aligned}
            g(p) & : T_{p}\cp(\cx) \simeq T_{p}^*\cp(\cx), \qquad f \mapsto \lb g(p)\lp f \rpz \rb,       \\
            g(p)^{-1} & : T_{p}^*\cp(\cx) \simeq T_{p}\cp(\cx), \qquad [f] \mapsto g(p)^{-1}\lp f \rpz .
        \end{aligned}
    \end{equation*}
    Thus the metric tensor can actually be viewed as an operator between these two spaces. The above tangent space $T_{p}\cp(\cx)$ is identified with the function space:
\begin{equation*}
    T_{p}\cp(\cx) \simeq C_0(\mathcal{X}) = \{f \in C(\cx) | \int_{\cx}fdx = 0\},
\end{equation*}
where $C(\cx)$ is the function space of continuous function on the space $\cx$. And its dual space $C(\mathcal{X})/\RR$, i.e. $f, h \in C(\mathcal{X}), f \sim h \text{ if } f = h + a, a \in \RR$, can be identified with the cotangent space of the density manifold: $T_{p}^*\cp(\cx) \simeq C(\mathcal{X})/\RR$. We use $\lb f \rb$ to represent the equivalent class of this function in $C(\mathcal{X})/\RR$. And the pairing between tangent spaces and cotangent spaces is merely $\la f,h \ra = \int_{\cx}fhdx$,
where we abuse the symbol $\langle , \rangle$ for the inner product. We note here that tangent spaces of a statistical model $\Theta$ can be viewed as subspaces of that of density manifold, i.e. we have inclusion:
\begin{equation*}
    T_{p}\Theta \hookrightarrow T_{p}\cp(\cx).
\end{equation*}
While taking the dual of this inclusion we get projection from the cotangent space of the density manifold to that of the statistical model:
\begin{equation*}
    T_{p}^*\cp(\cx) \rightarrow T_{p}^*\Theta.
\end{equation*}

An approach in information geometry is that one can reinterpret the metric tensor in the dual coordinates, i.e. cotangent space.
\begin{definition}[Score function]
Denote $\Phi_i$ $\colon \mathcal{X}\times\Theta\rightarrow \mathbb{R}, i = 1,...,n $ satisfying
$$\Phi_i(x;\theta)= \lb g(p) \lp \frac{\partial}{\partial \theta_i}p(x;\theta) \rpz \rb .$$
We call $\Phi_i$s score functions associated with the statistical information matrix $G$ and are equivalent classes in $C(\mathcal{X})/\RR$. The representatives in equivalent classes are determined by the following normalization condition:
\begin{equation}\label{normalization}
    \EE_{p_\theta} \Phi_i = 0,\qquad i = 1,...,n.
\end{equation}
Then the statistical information matrix satisfies  
\begin{equation*}
G(\theta)_{ij}=\int_{\mathcal{X}} \Phi_i(x;\theta)\Big(g(p_\theta)^{-1}\Phi_j\Big)(x;\theta)dx.
\end{equation*}
\end{definition}
\begin{remark}
    The normalization condition is an enforced condition. It fixes a representative for the score function in the equivalent class. And we assume that score functions are always integrable w.r.t. $p_{\theta}$.
\end{remark}
In above, there are two formulations of metric tensor, which use the following fact 
$g(p)^{-1}=g(p)^{-1}g(p)g(p)^{-1}$. Thus 
\begin{equation*}
\begin{split}
	G(\theta)_{ij}=& \ \Big\langle \nabla_{\theta_i} p_\theta, g( p_\theta) \nabla_{\theta_j} p_\theta \Big\rangle\\
	=& \ \Big\langle g(p_\theta)\nabla_{\theta_i} p_\theta, g( p_\theta)^{-1}  g(p_\theta)\nabla_{\theta_j} p_\theta\Big\rangle\\
=& \  \Big\langle \Phi_i, g(p_\theta)^{-1}\Phi_j\Big\rangle.
\end{split}
\end{equation*}

\begin{example}
    One important choice of metric is the Fisher-Rao metric:
    \begin{equation*}
        \begin{aligned}
            g(p) & : T_{p}\cp(\cx) \simeq T_{p}^*\cp(\cx), \qquad f \mapsto \lb \frac{f}{p} \rb,       \\
            g(p)^{-1} & : T_{p}^*\cp(\cx) \simeq T_{p}\cp(\cx), \qquad [f] \mapsto  p\lp f-E_p f \rpz .
        \end{aligned}
    \end{equation*}
    In this case, the statistical information matrix satisfies  
		\begin{equation*}
		G_{F}(\theta)_{ij}=\langle \frac{\partial}{\partial \theta_i} p_\theta, \frac{1}{ p_\theta}\frac{\partial}{\partial \theta_j} p_\theta \rangle=\int_\mathcal{X}\frac{\frac{\partial}{\partial \theta_i}p(x;\theta)\frac{\partial}{\partial \theta_j}p(x;\theta)}{p(x;\theta)}dx. 
		\end{equation*}
        And score functions of Fisher information matrix form
$$\Phi_i^F(x;\theta)=\frac{1}{p(x;\theta)}{\frac{\partial}{\partial \theta_i}p(x;\theta)}=\frac{\partial}{\partial \theta_i}\log p(x;\theta),$$ where the normalization condition holds automatically. In terms of score functions, the Fisher information matrix forms 
\begin{equation*}
\begin{split}
G_F(\theta)_{ij}=&\int_{\mathcal{X}}\Phi_i^F(x;\theta)\Big(g_F(p)^{-1}\Phi_j^F\Big)(x;\theta)dx\\
=&\int_{\mathcal{X}}\frac{\partial}{\partial \theta_i}\log p(x;\theta)\frac{\partial}{\partial \theta_j}\log p(x;\theta) p(x;\theta)dx\\
=&\ \mathbb{E}_{p_\theta}\lp \frac{\partial}{\partial \theta_i}\log p(x;\theta)\frac{\partial}{\partial \theta_j}\log p(x;\theta) \rpz.
\end{split}
\end{equation*}
In literature, $\Phi_i^F(x;\theta)=\frac{\partial}{\partial \theta_i}\log p(x;\theta)$ is named (Fisher) {\em score function}; while $G_F(\theta)$ is the {\em Fisher information matrix}. They play important roles in estimation, efficiency and Cramer-Rao bound. 
\end{example}
\begin{remark}
    The definition of Fisher score functions can be given in classical statistics as the gradient of the log-likelihood function w.r.t. parameters. Here we view it as an object on cotangent space associated with the Fisher-Rao metric on statistical models. That is, we have a family of canonical tangent vector fields $\frac{\pa}{\pa \theta_i} p_\theta$ on statistical models. Whenever there is a metric $g(p_\theta)$ on this manifold, we can define score functions associated with it as:
    \begin{equation*}
        \Phi_i\lp x; \theta \rpz = g(p_\theta)\frac{\pa}{\pa \theta_i} p_\theta\lp x; \theta \rpz.
    \end{equation*}
\end{remark}
From above fact, we observe that statistical concepts are related to the metric tensor in density manifold pull-back onto parameter space. In particular, classical statistics relates to the Fisher-Rao metric.  
The pull-back metric tensor forms an information matrix while dual variables define score functions. In this paper, we derive these notations in the other important statistical metric, known as the Wasserstein metric. 
\subsection{Wasserstein information matrix}
The other statistical metric, namely Wasserstein metric tensor forms 
\begin{equation*}
g_W(p)=(-\Delta_p)^{-1},\quad \textrm{where $\Delta_{ p}=\nabla\cdot( p\nabla)$.}
\end{equation*}
Here $\Delta_p$ is an elliptic operator weighted on a probability density $p$. When $p$ satisfies suitable conditions, standard PDE theory guarantees that the operators $\Delta_p^{-1}$ and $\Delta_p$ are an inverse to each other between function spaces:
\begin{equation*}
    \begin{aligned}
        \Delta_p^{-1} :& \   C_0(\mathcal{X}) \rightarrow C(\mathcal{X})/\RR ;       \\
        \Delta_p :& \  C(\mathcal{X})/\RR \rightarrow C_0(\mathcal{X}).
    \end{aligned}
\end{equation*}
The pull-back $G_W$ of $g_W$ is given by
	\begin{equation*}
	\begin{split}
	G_W(\theta)_{ij}= \langle \frac{\partial}{\partial \theta_i} p_\theta, (-\Delta_{p_\theta})^{-1}\frac{\partial}{\partial \theta_j} p_\theta\rangle.
         \end{split}
	\end{equation*}
Similar to the Fisher information matrix, we can rewrite $G_W$ by dual coordinates. Denote 
\begin{equation*}
\Phi_i^W(x;\theta)=(-\Delta_{p_\theta})^{-1}\frac{\partial}{\partial \theta_i}p(x;\theta).
\end{equation*}
Then 
\begin{equation*}
\begin{split}
G_W(\theta)_{ij}=& \ \langle \frac{\partial}{\partial \theta_i} p_\theta, (-\Delta_{p_\theta})^{-1}\frac{\partial}{\partial \theta_j} p_\theta\rangle\\
=& \ \langle \Phi_i^W, (-\Delta_{p_\theta}) \Phi_j^W\rangle\\
=& \ \int_{\mathcal{X}}(\nabla_x\Phi_i^W(x;\theta), \nabla_x\Phi_j^W(x;\theta))p(x;\theta)dx,
\end{split}
\end{equation*}
where the last equality holds by integration by parts w.r.t. $x$.

We summarize the above fact into the following definition. 
 \begin{definition}[Wasserstein information matrix \& score function]
 Denote $G_W(\theta)\in \mathbb{R}^{d\times d}$: 
 \begin{equation*}
G_W(\theta)_{ij}=\mathbb{E}_{p_\theta} \lb \nabla_x\Phi_i^W(x;\theta) \cdot \nabla_x\Phi_j^W(x;\theta)\rb,
\end{equation*}
where $\cdot$ refers to the inner product of vector and $\Phi_i^W \colon \mathcal{X}\times \Theta\rightarrow \mathbb{R}$ satisfies   
\begin{equation*}
    \begin{aligned}
        -\nabla_x\cdot(p(x;\theta)\nabla_x\Phi_i^W(x;\theta))=\frac{\partial}{\partial \theta_i}p(x;\theta), \quad 
        \EE_{p_\theta} \Phi_i^W = 0, \quad i = 1,2,...,d.
    \end{aligned}
\end{equation*}
We name functions $\Phi_i^W(x;\theta)=\Big((-\Delta_{p_\theta})^{-1}\frac{\partial}{\partial \theta_i} p_\theta\Big)(x;\theta)$ {\em Wasserstein score functions}, and call the matrix $G_W(\theta)$ the {\em Wasserstein information matrix}. 
  \end{definition}
  \begin{remark}\label{motivation}
This definition of information matrices is motivated by an intrinsic connection among distances, divergence functions, and metrics. 
            Specifically, given a smooth family of probability densities $p(x;\theta)$ and a given perturbation  $\Delta\theta\in T_\theta\Theta$,
            consider following Taylor expansions in term of $\Delta\theta$:
      \begin{equation} \label{motivation-equ}
          \begin{aligned}
            H(p(\theta)\| p(\theta+\Delta\theta)) & = \frac{1}{2}\Delta\theta^{\ts}G_F(\theta)\Delta\theta + o( (\Delta\theta)^2 ),       \\
            W_2( p(\theta+\Delta\theta), p(\theta))^2 & = \Delta\theta^{\ts}G_W(\theta)\Delta\theta + o( (\Delta\theta)^2).      
        \end{aligned}
  \end{equation}
          Here $H$ represents the Kullback--Leibler (KL) divergence or the relative entropy functional 
      \begin{equation*}
           H(p(\theta)\| p(\theta+\Delta\theta))=\int_{\mathcal{X}} p(x;\theta)\log\frac{p(x;\theta)}{p(x;\theta+\Delta\theta)}dx.
      \end{equation*}
    While $W_2^2$ denotes the squared $L^2$-Wasserstein distance defined by 
      \begin{equation}\label{linear-pro}
            W_2(p(\theta), p(\theta+\Delta\theta))^2 = \inf_{\pi \in \Pi\lp p(\theta), p(\theta+\Delta\theta) \rpz}\Big\{ \int_{\cx \times \cx} d_{\cx}\lp x, y \rpz^2 d\pi\lp x, y \rpz \Big\},
      \end{equation}
      where $\Pi\lp p(\theta), p(\theta+\Delta\theta) \rpz$ refers to the set of couplings between $p(\theta)$, $p(\theta+\Delta\theta)$ and $d_{\cx}$ is a distance function defined in $\cx$. Thus our approach parallels classical Fisher statistics. The Fisher information matrix approximates the KL divergence, which relates to the Fisher distance in Fisher geometry \cite{Amari1990_differentialgeometrical,ay2017information}, while WIM approximates the Wasserstein distance in Wasserstein geometry. Meanwhile, our approach can be viewed as exploring another aspect, namely metric aspect, of the Wasserstein statistics. For example, it can be related to the study of Wasserstein estimators \cite{2017arXiv170105146B}. 
  \end{remark}

We next study several basic properties of WIMs and score functions. We first illustrate a relation between Wasserstein and Fisher score functions.
\begin{proposition}[Poisson equation]\label{prop4}
Wasserstein score functions $\Phi^W_i(x;\theta)$ satisfy the following Poisson equation
\begin{equation}\label{ws}
\nabla_x \log p(x;\theta)\cdot \nabla_x\Phi^W_i(x;\theta)+\Delta_x\Phi^W_i(x;\theta)=-\frac{\partial}{\partial \theta_i}\log p(x;\theta).
\end{equation}
\end{proposition}
\begin{proof}
Notice the fact that 
\begin{equation*}
    \begin{aligned}
        \Big(\Delta_{p_\theta}\Big)\Phi_i^W(x;\theta) = & \  \nabla_x\cdot(p(x;\theta)\nabla_x\Phi_i^W(x;\theta))      \\
        = & \ \nabla_x p(x;\theta)\cdot \nabla_x\Phi_i^W(x;\theta)+p(x;\theta)\Delta_x\Phi_i^W(x;\theta).
    \end{aligned}
\end{equation*}
Then the Wasserstein score function $\Phi^W_i(x)$ satisfies 
\begin{equation*}
\nabla_xp(x;\theta)\cdot \nabla_x\Phi_i^W(x;\theta)+p(x;\theta)\Delta_x\Phi_i^W(x;\theta)=-\frac{\partial}{\partial \theta_i}p(x;\theta).
\end{equation*}
Divide the above equation on both sides by $p(x;\theta)$:
\begin{equation*}
\frac{1}{p(x;\theta)}\Big\{\nabla_xp(x;\theta)\cdot \nabla_x\Phi_i^W(x;\theta)+p(x;\theta)\Delta_x\Phi_i^W(x;\theta)\Big\}=-\frac{1}{p(x;\theta)}\frac{\partial}{\partial \theta_i}p(x;\theta),
\end{equation*}
i.e. 
\begin{equation*}
\frac{1}{p(x;\theta)}\nabla_xp(x;\theta)\cdot \nabla_x\Phi^W(x;\theta)+\Delta_x\Phi_i^W(x;\theta)=-\frac{1}{p(x;\theta)}\frac{\partial}{\partial \theta_i}p(x;\theta).
\end{equation*}
Since $\frac{1}{p(x;\theta)}\nabla_xp(x;\theta)=\nabla_x\log p(x;\theta)$ and $\frac{1}{p(x;\theta)}\frac{\partial}{\partial \theta_i}p(x;\theta)=\frac{\partial}{\partial \theta_i}\log p(x;\theta)$, 
we prove the property \eqref{ws}.
\end{proof}
We then demonstrate that Wasserstein score functions and information matrices can also be decomposed into a summation of separable functions in independent models.  
\begin{proposition}[Separability]\label{sep}
If $p(x;\theta)$ is an independence model, i.e. 
\begin{equation*}
    p(x; \theta) = \Pi_{k = 1}^n p_k(x_k;\theta), \quad x_k \in \cx_k,\quad x=(x_1,\cdots, x_n).
\end{equation*}
Then there exists a set of functions $\Phi_i^{W,k}\colon \mathcal{X}_k \times\Theta \rightarrow \mathbb{R}, i = 1,2, \cdots, \dim \Theta, k = 1, 2, \cdots, n$, such that 
\begin{equation}\label{linear}
\Phi_i^W(x;\theta)=\sum_{k = 1}^n\Phi_i^{W,k}(x_k;\theta).
\end{equation}
In addition, the WIM is separable:
\begin{equation*}
G_W(\theta) =\sum_{k = 1}^nG_{W}^k(\theta),
\end{equation*}
where $\lp G_{W}^k(\theta) \rpz_{ij} =\mathbb{E}_{p_k} \lp \nabla_{x}\Phi_i^{W,k}(x;\theta), \nabla_{x}\Phi_j^{W,k}(x;\theta) \rpz $.
\end{proposition}
\begin{proof}
The proof follows from proposition \ref{prop4}. Suppose one can write the solution in form of \eqref{linear}, then equation \eqref{ws} forms 
\begin{equation*}
\sum_{k=1}^n \Big\{\nabla_{x_k}\log p_k(x_k;\theta_k)\nabla_{x_k}\Phi_i^{W,k}(x_k;\theta)+\Delta_{x_k}\Phi_i^{W,k}(x_k;\theta)-\frac{\partial}{\partial \theta_i}\log p_k(x_k;\theta)\Big\}=0.
\end{equation*}
From the separable method for solving the Poisson equation, we derive
\begin{equation*}
\nabla_{x_k}\log p_k(x_k;\theta_k)\nabla_{x_k}\Phi_i^{W,k}(x_k;\theta)+\Delta_{x_k}\Phi_i^{W,k}(x_k;\theta_k)-\frac{\partial}{\partial \theta_i}\log p_k(x_k;\theta)=0.
\end{equation*}
We finish the first part of the proof. In addition, 
\begin{equation*}
\begin{split}
\lp G_{W}(\theta) \rpz_{ij}=& \ \mathbb{E}_{p_\theta} \lp \nabla_x\Phi_i^W(x;\theta), \nabla_x\Phi_j^W(x;\theta) \rpz\\
=& \ \mathbb{E}_{p_\theta} \lp \sum_{k} (\nabla_{x}\Phi_i^{W,k}(x;\theta), \nabla_{x}\Phi_j^{W,k}(x;\theta) \rpz\\
=& \ \sum_{k}\mathbb{E}_{p_k} \lp \nabla_{x}\Phi_i^{W,k}(x;\theta), \nabla_{x}\Phi_j^{W,k}(x;\theta) \rpz\\
=& \ \sum_{k}\lp G_{W}^k(\theta) \rpz_{ij}.
\end{split}
\end{equation*}
\end{proof}

We next list some analytical solutions for WIMs and score functions in 1-d case. See related studies in \cite{villani2003topics} (c.f. Ch 2.2).
\begin{proposition}[One-dimensional sample space]
If $\mathcal{X}\subset \mathbb{R}^1$, Wasserstein score functions satisfy
\begin{equation}\label{1d}
\Phi_i^W(x;\theta)=-\int_{\cx \cap \lp \infty , x \rb} \frac{1}{p(z;\theta)}\frac{\partial}{\partial\theta_i}F(z;\theta)dz, 
\end{equation}
 where $F(x;\theta)=\int_{\cx \cap \lp \infty , x \rb} p(y;\theta)dy$ is the cumulative distribution function. And the WIM satisfies 
 \begin{equation*}
 G_W(\theta)_{ij}=\mathbb{E}_{p_\theta} \lp \frac{\frac{\partial}{\partial\theta_i}F(x;\theta)\frac{\partial}{\partial\theta_j}F(x;\theta)}{p(x;\theta)^2} \rpz.
 \end{equation*}
 \end{proposition}
If the dimension of sample space $\mathcal{X}$ is larger than $1$, exact solutions of Wasserstein score functions and information matrices depend on solutions of Poisson equation \eqref{ws}. We leave the derivation of general formulas for interested readers.

\subsection{Analytic examples}
We present several analytical examples of the WIM in one-dimensional sample space. The derivation is given in section \ref{supp-2}. 

\begin{example}[Gaussian distribution]
 Consider Gaussian distribution families with mean value $\mu$ and standard variance $\sigma > 0$, i.e.
$p(x;\mu,\sigma)=\frac{1}{\sqrt{2\pi} \sigma}e^{-\frac{1}{2\sigma^2}(x-\mu)^2}$.
Wasserstein score functions satisfy
\begin{equation*}
\begin{split}
\Phi^W_\mu(x;\mu,\sigma) = x - \mu ,\quad\Phi^W_\sigma(x;\mu,\sigma)=\frac{(x-\mu)^2 - \sigma^2}{2\sigma}.
\end{split}
\end{equation*}
And the WIM satisfies 
\begin{equation*}
G_W(\mu,\sigma)=\begin{pmatrix}
1& 0 \\
0 &  1
\end{pmatrix}.
\end{equation*}
\end{example}
\begin{example}[Exponential distribution]
Consider exponential distribution families $Exp(m,\lambda)$, i.e.
$p(x;m,\lambda)=
\textbf{1}_{[m, \infty)}(x)\lambda e^{-\lambda (x - m)}$, where the function $\textbf{1}_C$ is the indicator function for a set $C \subset \mbR$. Wasserstein score functions satisfy
\begin{equation*}
\Phi^W_\lambda(x;m,\lambda)=\frac{(x - m)^2 - \frac{2}{\lambda^2}}{2\lambda}, \qquad \Phi^W_m(x;m,\lambda)= x - m -\frac{1}{\lambda}.
\end{equation*}
And the WIM satisfies 
\begin{equation*}
G_W(m,\lambda)=\begin{pmatrix}
    1 & \frac{1}{\lambda^2}     \\
    \frac{1}{\lambda^2} & \frac{2}{\lambda^4}
\end{pmatrix}.
\end{equation*}
\end{example}

\begin{example}[Laplacian distribution]
Consider Laplacian distribution families $La(m,\lambda)$, i.e. $p(x;m,\lambda) = \frac{\lambda}{2}e^{-\lambda|x-m|}$. Wasserstein score functions satisfy 
\begin{equation*}
\Phi^W_\lambda(x;m,\lambda)=\frac{(x - m)^2 - \frac{2}{\lambda^2}}{2\lambda}, \qquad \Phi^W_m(x;m,\lambda)= x - m.
\end{equation*}
Notice that score functions for exponential families and Laplacian families have similar formulas. And the WIM satisfies 
\begin{equation*}
G_W(m,\lambda)=\begin{pmatrix}
    1 & 0     \\
    0 & \frac{2}{\lambda^4}
\end{pmatrix}.
\end{equation*}
We will show below that the Laplacian family has an advantage that densities within this family have the same support. Thus this model is convenient for us to compare the WIM with the Fisher information matrix. See details in section \ref{ex-func}.
\end{example}

\begin{example}[Uniform distribution]
Consider uniform distribution families within interval $[a,b]$, i.e. $p(x;a,b)= 
\frac{1}{b-a} \textbf{1}_{[a, b]}(x)$. Wasserstein score functions satisfy
\begin{equation*}
    \begin{aligned}
        \Phi^W_a(x;a,b) & = \frac{x(a+b-x)}{(b-a)} - \frac{b^2 + a^2 + 4ab }{6},  \\
        \Phi^W_b(x;a,b) & = \frac{b(x-2a)}{(b-a)} - \frac{b^2 - 3ab}{2} .
    \end{aligned}
\end{equation*}
And the WIM satisfies 
\begin{equation*}
G_W(a,b) = \frac{1}{3} \begin{pmatrix}
1 & \frac{1}{2}\\
\frac{1}{2} &  1
\end{pmatrix}.
\end{equation*}
\end{example}

\begin{example}[Wigner semicircle distribution]
Consider semicircle distribution families, i.e. $p(x;m,R)=\textbf{1}_{[-R + m, R + m]}(x)
\frac{2}{\pi R^2}\sqrt{R^2 - \lp x - m \rpz^2}$. Wasserstein score functions satisfy
\begin{equation*}
\Phi^W_R(x;m,R)=\frac{1}{R}(\frac{\lp x - m \rpz^2}{2}-\frac{R^2}{8}), \qquad \Phi^W_p(x;m,R)= x - m.
\end{equation*}
And the WIM satisfies 
\begin{equation*}
G_W(m,R)=\begin{pmatrix}
    1 & 0     \\
    0 & \frac{1}{4}
\end{pmatrix}.
\end{equation*}
\end{example}

\begin{example}[Independent model]
Consider an independent model as follow: suppose $X\sim p_1(x;\theta)$, and $Y\sim p_2(x;\theta)$, and $(X,Y)\sim p(x,y;\theta)$, then
\begin{equation*}
p(x,y;\theta)=p_1(x;\theta)p_2(y;\theta).
\end{equation*}
Denote Wasserstein score functions (resp. WIM) for statistical model $X\sim p_1(x;\theta), Y\sim p_2(x;\theta)$ as $\Phi^W_1(x;\theta), \Phi^W_2(x;\theta)$($G_W^1(x;\theta), G_W^2(x;\theta)$) respectively. Then, Wasserstein score functions for this model $(X,Y)\sim p(x,y;\theta)$ satisfy
\begin{equation*}
\Phi^W(x,y;\theta)=\Phi^W_1(x;\theta)+\Phi^W_2(y;\theta),
\end{equation*}
because of the additivity of expectation $\EE_{p_\theta}\Phi^W(x,y;\theta)=\EE_{p_\theta}\Phi^W_1(x;\theta) + \EE_{p_\theta}\Phi^W_2(y;\theta) = 0$. And the WIM satisfies 
\begin{equation*}
G_W(\theta)=G_W^1(\theta)+G_W^2(\theta).
\end{equation*}
The proof follows directly from proposition \ref{sep}.
\end{example}
\par
In above discussions, all examples are based on location-scale families, which will be derived carefully in section \ref{location}. 
We show that location-scale families are totally geodesic submanifolds in Wasserstein geometry.
\subsection{WIM in generative models}
\par
In this section, we study the WIM for generative models using ReLU function, which is given by
\begin{equation*}
    \sigma \lp x \rpz = \lbb \begin{aligned}
        & 0, \quad x \leq 0,    \\
        & x, \quad x > 0.
    \end{aligned}\right.
\end{equation*}
Generative models are powerful in machine learning \cite{10.1007/978-3-030-26980-7_74}. It applies the reparameterization trick (known as push-forward relation) to conduct efficient sampling. In practice, one often applies the ReLU as a push-forward function \eqref{push-forward}. For this reason, we call this kind of models ReLU push-forward family. The push-forward measure $f_*p$ is defined as
\begin{equation}\label{push-forward}
    \int_{A}f_* pdx = \int_{f^{-1}(A)}pdx, \quad \forall  A \subset \mbR.
\end{equation}

\par
To keep derivations simple, we consider one-dimensional cases with a given distribution $p_0\lp x \rpz$, $x\in\mathbb{R}$. And its cumulative distribution function is denoted by $F_0\lp x \rpz$.
\begin{example}[ReLU push-forward family]\label{ReLU-ex}
We use a family of ReLU functions $f_{\theta}$ parameterized by $\theta$ to generate a push-forward family
\bequn
	\begin{aligned}
		p : \Theta \simeq \RR \rightarrow & \cp \lp \RR \rpz: \quad \theta \mapsto p_{\theta},		\\
		p_{\theta}\lp x \rpz = p\lp x; \theta \rpz= \lp f_{\theta*} p_0 \rpz \lp x \rpz, \quad & f_{\theta} \lp x \rpz= \ \sigma\lp x - \theta \rpz= \lbb
		\begin{aligned}
		& 0, \qquad \qquad x \leq \theta,			\\
		& x - \theta, \ \qquad x > \theta.
		\end{aligned}\right.
	\end{aligned}
\eequn
The WIM of $p_{\theta}$ satisfies
\bequ\label{ReLU-metric1}
	G_W \lp \theta \rpz= 1 - F_0\lp \theta \rpz.
\eequ
\par
We can also consider another family of ReLU maps to push forward the source distribution. This family is given by
\bequn
	\begin{aligned}
		p : \Theta \simeq \mbR \rightarrow & \mcP\lp \mbR \rpz: \quad \theta \mapsto p_{\theta}		\\
		p_{\theta}\lp x \rpz = p\lp x; \theta \rpz= \lp h_{\theta*} p_0 \rpz \lp x \rpz, \qquad & h_{\theta}\lp x \rpz= \ \sigma\lp x - \theta \rpz+ \theta = \lbb
		\begin{aligned}
		& \theta, \qquad x \leq \theta,			\\
		& x, \qquad x > \theta.
		\end{aligned}\right.
	\end{aligned}
\eequn
The WIM of $p_{\theta}$ satisfies
\bequ\label{ReLU-metric2}
	G_W \lp \theta \rpz= F_0\lp \theta \rpz.
\eequ 
A figure illustrating these two families is provided below.
\begin{figure}[H]
  \centering
  \centerline{\includegraphics[width=0.52\linewidth]{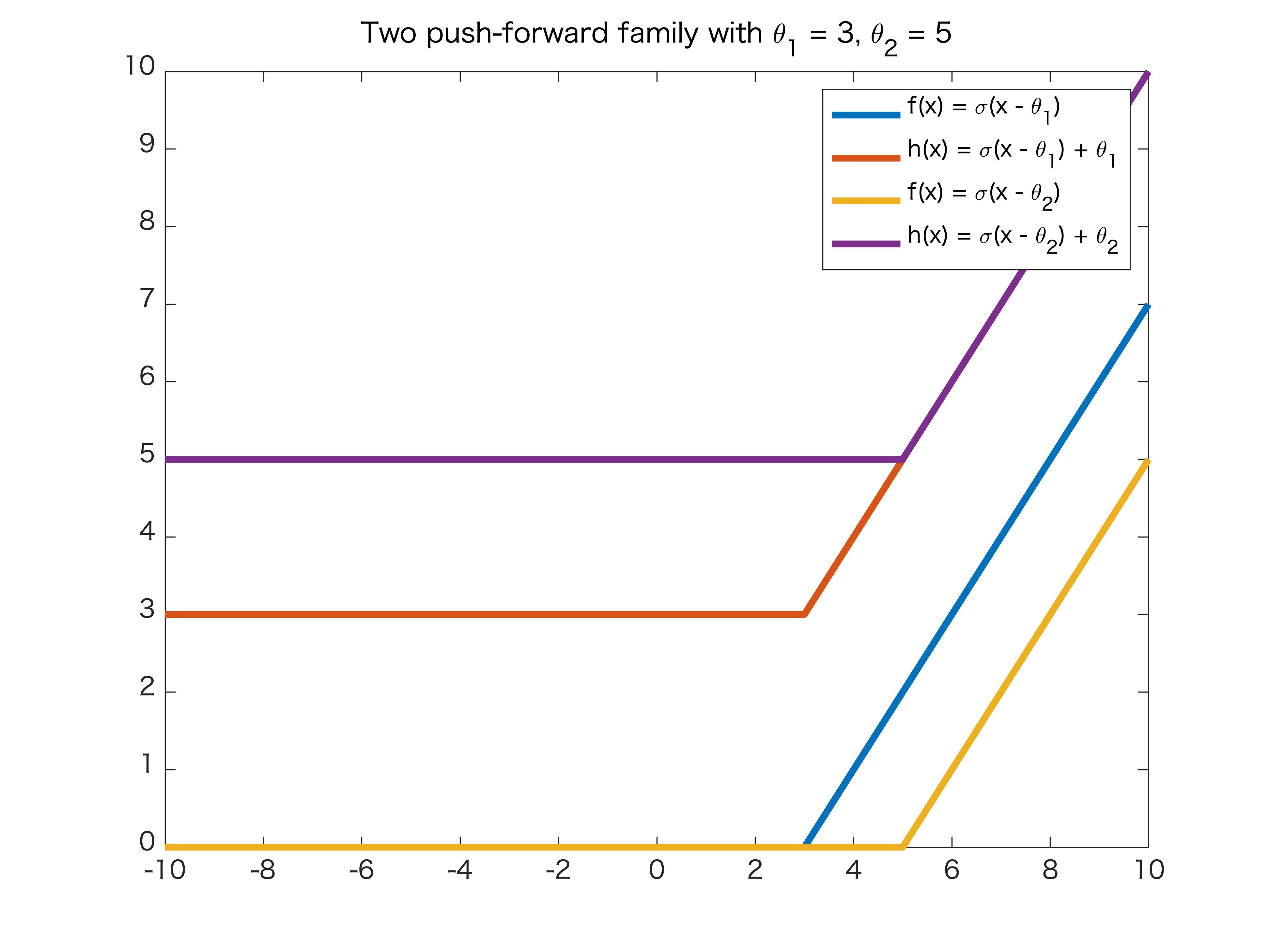}}
  \caption{This figure plots two examples of push-forward families we described above with parameters chosen as $\theta_1 = 3, \theta_2 = 5$.}
\end{figure}
\begin{remark}
    To calculate the WIMs of this model, we cannot use previous approaches of score functions, since it is not smooth enough. Instead, we utilize the idea stated in remark \ref{motivation}. Namely, we use the relation \eqref{motivation-equ} between Wasserstein distance and the WIM to compute the latter.
\end{remark}
\begin{proof}

Consider the following two push-forward distributions given by
\bequn
	\begin{aligned}
	    \lp f_{\theta + \Delta \theta *} p_0 \rpz \lp x \rpz& = F_0\lp \theta + \Delta \theta \rpz\delta_0 + p_0\lp \cdot + \theta + \Delta \theta \rpz_{\lb 0, \infty \rpz},           \\
		\lp f_{\theta*} p_0 \rpz \lp x \rpz& = F_0\lp \theta \rpz\delta_0 + p_0\lp \cdot + \theta \rpz_{\lb 0, \infty \rpz},
	\end{aligned}
\eequn
where $\delta_0$ refers to the Dirac measure concentrating at point $0$. And $p_0\lp \cdot + \theta \rpz_{\lb 0, \infty \rpz}$ represents the measure $\wtd p \lp x \rpz= p_0 \lp x + \theta \rpz$ restricting to the interval $\lb 0, \infty \rpz$. Using monotonicity of transportation plan in 1-d, we conclude that its restriction on $\lp 0, \infty \rpz$ transports measure on $x$ to $x + \Delta \theta$. And it remains to transport the Dirac measure centered at $0$ to the remained place. The transportation cost is given by
\bequ\label{ReLU-1}
	\begin{aligned}
		W_2^2\lp f_{\theta*} p_0, f_{\theta + \Delta \theta*} p_0 \rpz= & \  \int_0^{\infty} p_0\lp x + \theta + \Delta\theta \rpz \lp \Delta\theta \rpz^2 dx + \int_0^{\Delta\theta} x^2 p_0\lp x + \theta \rpz dx 			\\
		= & \ \lp \Delta\theta \rpz^2\lp 1 - F_0 \lp \theta + \Delta \theta \rpz \rpz + O \lp \lp \Delta \theta \rpz^3 \rpz,
	\end{aligned}
\eequ
where the third equality holds by
\bequn
	\int_{0}^{\Delta \theta} p_0\lp x + \theta \rpz dx = O \lp \Delta \theta \rpz.
\eequn
Notice in formula \eqref{ReLU-1}, we decompose the transportation cost into two parts: the first one is concerned with the cost on the right part of $0$, while the second one considers transporting Dirac measure at $0$ to the remained part. Since the WIM is an infinitesimal approximation of the Wasserstein distance, i.e. equation \eqref{motivation-equ}. The conclusion \eqref{ReLU-metric1} follows.
\par
For the other family, derivations follow the same method as before. Specifically, we have
\bequn
	\begin{aligned}
		W_2^2\lp h_{\theta*} p_0, h_{\theta + \Delta \theta*} p_0\rpz
		= & \ \int_{0}^{\Delta \theta} x^2 p_0\lp x + \theta \rpz dx + \lp \Delta \theta \rpz^2 F_0\lp \theta \rpz		\\
		= & \lp \Delta \theta \rpz^2 F_0\lp \theta \rpz+ O \lp \lp \Delta \theta \rpz^3 \rpz,
	\end{aligned}
\eequn
where we again decompose the transportation cost into two parts. The first one is absolutely continuous w.r.t the Lebesgue measure, while the second one contains a Dirac measure.
\end{proof}
\par
Here we notice that density functions in ReLU push-forward family can be singular. 
Thus the Fisher information matrix, which depends on an explicit formula of density functions, namely
\begin{equation*}
    G_F(\theta)_{ij} =
\int_{\mathcal{X}}\frac{\partial}{\partial \theta_i}\log p(x;\theta)\frac{\partial}{\partial \theta_j}\log p(x;\theta) p(x;\theta)dx
\end{equation*}
fails to exist in these models. On the contrary, as we have shown in the above example, the WIM still exists. This property shows that the WIM can provide statistical studies for generative models, while the Fisher information matrix in classical statistics can not.
\end{example}

\section{Wasserstein estimation}\label{sec-WCR}
\par
In this section, we define the Wasserstein covariance and establish the Wasserstein-Cramer-Rao bound. Based on these concepts, we introduce a notion of efficiency in Wasserstein statistics. Several examples based on the previous section are provided.
\subsection{Estimation and efficiency}
\par
Following the spirit under which we introduce information matrices in section \ref{sec2}, we generalize the definition of covariance matrix for a given metric tensor $g$ on probability space.

Denote $\langle f, h\rangle_{g}$ the inner product of cotangent vectors $f,h$ in the metric $g$.
\begin{equation*}
\langle f, h\rangle_{g}= \langle f, g(p)^{-1} h\rangle. 
\end{equation*}
\begin{definition}[Information covariance matrix]
Given a statistical model $\Theta$ with metric $g$, and statistics $T$, $\wtd T$ which are of dimension $m,n$ respectively, the information covariance matrix of $T$,$\wtd T$ associated to metric $g$ is defined as:
\begin{equation*}
\mathrm{Cov}^g_\theta[T, \wtd T]_{ij} = \la T_i,\wtd T_j \ra_g ,
\end{equation*}
where $T_i$, $\wtd T_j$ are random variables as function of $x$. 
Denote the information variance as: 
\begin{equation*}
\mathrm{Var}^g_\theta[T] =  \la T, T \ra_g .
\end{equation*}
\end{definition}
\begin{remark}
Here the inner product $\la T_i,\wtd T_j \ra_g$ is obtained by viewing the statistics as cotangent vectors on density manifold $\cp(\cx)$. 
\end{remark}
\begin{example}[Fisher covariance]
\par
Given two statistics $T_1,T_2$, we view them as cotangent vectors in space $C(\cx)/\RR$. Hence their Fisher inner product is defined as
\begin{equation*}
    \la T_1, T_2 \ra_{g_F} = \int_{\cx} \lp T_1 - \EE_{p_\theta}\lb T_1 \rb \rpz\lp T_2 - \EE_{p_\theta}\lb T_2 \rb \rpz p_\theta dx.
\end{equation*}
Here choosing the function $T_1 - \EE_{p_\theta}\lb T_1 \rb$ as the representative of $[T_1]$ is consistent with the normalization requirement \eqref{normalization}. Thus Fisher covariance (resp. variance) reduces to the original definition of the covariance (resp. variance) in probability theory.
\par
And the classical Cramer-Rao bound is given by
\begin{equation*}
\mathrm{Cov}^F_\theta[T(x)] \succeq \nabla_\theta \mathbb{E}_{p_\theta} [T(x)]^{\ts}G_F(\theta)^{-1}\nabla_\theta \mathbb{E}_{p_\theta} [T(x)],
\end{equation*}
where $G_F(\theta)$ is the Fisher information matrix.
In 1-d cases, the above forms 
\begin{equation*}
\mathrm{Var}_\theta[T(x)] \geq 
\frac{\lp \nabla_\theta \mathbb{E}_{p_\theta}T\lp x \rpz \rpz^2}{G_F(\theta)}.
\end{equation*}
\end{example}
\par
We next focus on the Wasserstein covariance operator.
\begin{definition}[Wasserstein covariance]
Given a statistical model $\Theta$, denote the Wasserstein covariance as follows:
\begin{equation*}
\mathrm{Cov}^W_\theta[T_1, T_2]=\mathbb{E}_{p_\theta} \lp \nabla_x T_1(x), \nabla_x T_2(x)^{\ts} \rpz,
\end{equation*}
where $T_1$, $T_2$ are random variables as functions of $x$ and the expectation is taken w.r.t. $x \sim p_\theta$. 
Denote the Wasserstein variance: 
\begin{equation*}
\mathrm{Var}^W_\theta[T] = \mathbb{E}_{p_\theta} \lp \nabla_x T(x), \nabla_x T(x)^{\ts} \rpz.
\end{equation*}
\end{definition}
\begin{theorem}[Wasserstein-Cramer-Rao inequality]\label{WCR}
Given any set of statistics $T = \left( T_1,...,T_m\right) \colon $ \\ $\mathcal{X}\rightarrow \mathbb{R}^m$, where $m$ is the number of the statistics, define two matrices $\mathrm{Cov}^W_\theta[T(x)]$, $ \nabla_\theta \mathbb{E}_{p_\theta} [T(x)]^{\ts}$ as below:
\begin{equation*}
    \mathrm{Cov}^W_\theta[T(x)]_{ij} = \mathrm{Cov}^W_\theta[T_i,T_j], \qquad \nabla_\theta \mathbb{E}_{p_\theta} [T(x)]^{\ts}_{ij} = \frac{\pa}{\pa \theta_j} \mathbb{E}_{p_\theta} [T_i(x)],
\end{equation*}
then
\begin{equation*}
\mathrm{Cov}^W_\theta[T(x)] \succeq \nabla_\theta \mathbb{E}_{p_\theta} [T(x)]^{\ts}G_W(\theta)^{-1}\nabla_\theta \mathbb{E}_{p_\theta} [T(x)],
\end{equation*}
where the notion $\succeq$ refers to that the difference of two matrices is positive semi-definite.
\end{theorem}
\begin{proposition}[Covariance property]\label{Cov}
Given the Wasserstein score function $\Phi^W_i(x;\theta)$ and any smooth statistic $f\colon \mathcal{X}\rightarrow \mathbb{R}$, then 
\begin{equation*}
\frac{\partial}{\partial \theta_i} \mathbb{E}_{p_\theta} f(x) = \mathbb{E}_{p_\theta} (\nabla_x\Phi^W_i(x; \theta), \nabla_{x} f(x)) = \la \Phi^W_i, f \ra_{g_W}.
\end{equation*}
\end{proposition}
\begin{proof}
Notice the fact that  
\begin{equation*}
\begin{split}
\frac{\partial}{\partial \theta_i} \mathbb{E}_{p_\theta} f(x) = & \frac{\partial}{\partial \theta_i}\int_{\mathcal{X}}f(x) p(x;\theta)dx\\
=&\int_{\mathcal{X}}f(x)\frac{\partial}{\partial \theta_i}p(x;\theta)dx\\
=&\int_{\mathcal{X}}f(x) \Big(-\nabla_x\cdot(p(x;\theta)\nabla_x\Phi^W_i(x;\theta))\Big)dx\\
=&\int_{\mathcal{X}}\nabla_xf(x)\cdot \nabla_x\Phi^W_i(x;\theta) p(x;\theta)dx,
\end{split}
\end{equation*}
where the third equality comes from the definition of Wasserstein score functions, while the last equality holds by integration by parts formula in spatial domain. 
\end{proof}
\begin{remark}
This property is in contrast to Fisher score functions
\begin{equation*}
    \begin{aligned}
& \ \frac{\partial}{\partial \theta_i} \mathbb{E}_{p_\theta} f(x) = \mathbb{E}_{p_\theta} \lp f(x) \frac{\partial}{\partial \theta_i}\log p(x;\theta) \rpz \\
= & \  \textrm{Cov}_\theta[( f(x), \frac{\partial}{\partial \theta_i}\log p(x;\theta))] = \la \Phi^F_i, f \ra_{g_F}.
    \end{aligned}
\end{equation*}
This is merely a dual relation between tangent and cotangent space in the density manifold.
\end{remark}
\begin{proof}[Proof of Theorem \ref{WCR}]
\par
By the definition of semi-positive matrix, it suffices to prove that for arbitrary $v \in \RR^m$, we have:
\begin{equation*}
    v^{\ts}\mathrm{Cov}^W_\theta[T(x)]v \geq v^{\ts}\nabla_\theta \mathbb{E}_{p_\theta} [T(x)]^{\ts}G_W(\theta)^{-1}\nabla_\theta \mathbb{E}_{p_\theta} [T(x)]v.
\end{equation*}
Here we define $T_v = v^{\ts}T$ as the statistic associated to the vector $v$. Then the LHS of above formula equals to the variance of $T_v$:
\begin{equation*}
    v^{\ts}\mathrm{Cov}^W_\theta[T(x)]v = \mathrm{Var}^W_\theta[T_v].
\end{equation*}
\par
As we have mentioned before, score functions $\Phi^W_i$s, as a set of basis, span a linear space $V_{p\lp x;\theta \rpz}^* \Theta$ of the cotangent space $T_{p\lp x;\theta \rpz}^* \cp\lp \cx \rpz$ at each point $p\lp x;\theta\rpz$ on the density manifold. Meanwhile, the statistic $T_v \colon \mathcal{X}\rightarrow \mathbb{R}$ can be viewed as a cotangent vector field on this statistical model. Now, the subspace $V_{p\lp x;\theta \rpz}^* \Theta$ at each point $\theta$ is a finite-dimensional subspace of the Hilbert space $T_{p(x;\theta)}^*\mathcal{P}(\mathcal{X})$ endowed with the Wasserstein inner product. Thus it is a closed linear subspace. By elementary theory of functional analysis, we have orthogonal projection operator $\P$:
\begin{equation*}
    \P \colon T_{p(x;\theta)}^*\mathcal{P}(\mathcal{X}) \rightarrow V_{p\lp x;\theta \rpz}^* \Theta.
\end{equation*}
Since $\Phi^W_i$s span the whole subspace, we have:
\begin{equation*}
    \la \Phi^W_i, v - \P v \ra_{g_W} = 0, \qquad \forall v \in T_{p(x;\theta)}^*\mathcal{P}(\mathcal{X}).
\end{equation*}
\par
Now, back to the theorem, we have:
\begin{equation*}
    \mathrm{Var}^W_\theta[T_v] = \mathbb{E}_{p_\theta} \lb (\nabla_x T_v(x), \nabla_x T_v(x)^{\ts}) \rb = \la T_v, T_v \ra_{g_W} \geq \la \P T_v, \P T_v \ra_{g_W},
\end{equation*}    
where the last inequality holds by the property of the orthogonal projection operator.        \\
\indent
Now, since $\P$ is the projection onto the subspace $V_{p\lp x;\theta \rpz}^* \Theta$ with a set of basis $\Phi^W_i$, at each point $\theta$, we can write the cotangent vector $\P T_v$ as a linear combination of Wasserstein score functions:
\begin{equation*}
    \P T_v = \sum_{i = 1}^d t_i^{\theta} \Phi^W_i,
\end{equation*}
where the superscript of $t_i^{\theta}$ indicates the dependency on point $\theta$. Now, plugging this linear combination into the Wasserstein metric, we get:
\begin{equation*}  
    \begin{aligned}
        \la \P T_v, \P T_v \ra_{g_W} & = \sum_{i = 1}^d t_i^{\theta} \la \P T_v, \Phi^W_i \ra_{g_W}         \\
        & = \sum_{i,k = 1}^d t_k^{\theta} \delta_{i}^{k} \la \P T_v, \Phi^W_i \ra_{g_W}       \\
        & = \sum_{i,j,k = 1}^d t_k^{\theta} g_{kj} g^{ij}  \la \P T_v, \Phi^W_i \ra_{g_W}            \\
        & = \sum_{i,j = 1}^d \la \P T_v, \Phi^W_i \ra_{g_W} \lp G_W(\theta)^{-1} \rpz_{ij}  \la \P T_v, \Phi^W_i \ra_{g_W}     \\
        & = \nabla_\theta\mathbb{E}_{p_\theta} [T_v(x)]^{\ts}G_W(\theta)^{-1}\nabla_\theta\mathbb{E}_{p_\theta} [T_v(x)],
    \end{aligned}
\end{equation*}
where $G_W(\theta)^{-1}$ is the inverse matrix of the WIM, $g_{kj}$, $g^{ij}$ are elements of matrix $G_W$, $G_W^{-1}$ respectively, and the third equality holds by the fact $\sum_{j}g_{kj}g^{ij} = \delta_{i}^k$. The last equality is guaranteed by proposition \ref{Cov}.
Combining the above calculation and the comparison between the inner product of $T_v$ and $\P T_v$, we obtain the desired result.
\end{proof}
Given the above theorem, we can define the Wasserstein efficiency as follows. 
\begin{definition}
    For an estimator $T(x)$, it is Wasserstein efficient if and only if it attains the Wasserstein-Cramer-Rao bound, namely:
    \begin{equation*}
\mathrm{Var}^W_\theta[T(x)] = \nabla_\theta \mathbb{E}_{p_\theta} [T(x)]^{\ts}G_W(\theta)^{-1}\nabla_\theta \mathbb{E}_{p_\theta} [T(x)]. 
\end{equation*}
\end{definition}
\begin{remark}
    From the above derivation, a sufficient and necessary condition for a statistic to be efficient is that, it can be written as a linear combination of score functions. Notice this criterion is valid for any metrics, including both Fisher and Wasserstein metrics. This is a purely geometric condition and we seek below in various statistical models to find out its statistical significance.
\end{remark}
\begin{remark}
As shown in the above theorem, if we denote the Fisher-Rao metric as $g_F(p)=\frac{1}{p}$, we then derive the classical Cramer-Rao bound:
\begin{equation*}
\mathrm{Cov}_\theta(T(x), T(x))\geq \nabla_\theta\mathbb{E}_{p_\theta} [T(x)]^{\ts}G_{F}(\theta)^{-1}\nabla_\theta\mathbb{E}_{p_\theta} [T(x)].
\end{equation*}
Here the Fisher-Rao metric corresponds to the classical covariance operator
\begin{equation*}
\mathrm{Cov}_\theta(T(x), T(x))=\mathbb{E}_{p_\theta}[(T(x)-\mathbb{E}_{p_\theta}T(x), T(x)-\mathbb{E}_{p_\theta}T(x))],
\end{equation*}
which depends on the expectation of statistics $\mathbb{E}_{p_\theta} T(X)$. Furthermore, given any information matrix on a statistical model, we have an associated Cramer-Rao bound. 
\end{remark}
\subsection{Analytic examples}
\begin{example}[Gaussian distribution]Recall that given a Gaussian distribution with mean value $\mu$ and standard variance $\sigma$, Wasserstein score functions satisfy
\begin{equation*}
\begin{split}
\Phi_\mu^W(x;\mu,\sigma) = x - \mu , \quad\Phi_\sigma^W(x;\mu,\sigma)=\frac{(x-\mu)^2 - \sigma^2 }{2\sigma},
\end{split}
\end{equation*}
with the WIM 
\begin{equation*}
G_W(\mu,\sigma)=\begin{pmatrix}
1 & 0 \\
0 &  1
\end{pmatrix}.
\end{equation*}
Thus by the criterion, we know that efficient statistics in Wasserstein cases are exactly those which can be written as linear combinations of score functions. Since statistics only depend on samples $x_i$ and do not depend on parameters $\mu,\sigma$, they must be of the form:
\begin{equation*}
    ax^2 + bx + c = 2a\sigma \Phi_\sigma^W + (2a\mu + b)\Phi_\mu^W + c + a\mu^2 + b\mu - a\sigma^2,
\end{equation*}
since Wasserstein cotangent vectors are determined up to a constant. Wasserstein efficient statistics are degree 2 polynomials of $x$.
\par
While score functions for Fisher case are given by:
\begin{equation*}
\begin{split}
\Phi_\mu^F(x;\mu,\sigma) = \frac{x - \mu}{\sigma^2}, \quad \Phi_\sigma^F(x;\mu,\sigma) = \frac{(x-\mu)^2}{\sigma^3} - \frac{1}{\sigma}.
\end{split}
\end{equation*}
And the Fisher information matrix satisfies 
\begin{equation*}
G_F(\mu,\sigma) = \begin{pmatrix}
			\frac{1}{\sigma^2} & 0		\\
			0 & \frac{2}{\sigma^2}
		\end{pmatrix}.
\end{equation*}
Thus we conclude that although we have different score functions in Wasserstein and Fisher-Rao cases, it turns out that efficient statistics associated with these two information matrices coincide. But still, Fisher and Wasserstein information matrices provide us with different Cramer-Rao bounds. The Fisher-Cramer-Rao bound is better when we have prior knowledge that $\sigma$ is large while worse if $\sigma$ is small.

\end{example}
\begin{example}[Exponential distribution]
Given an exponential distribution, Wasserstein score functions satisfy
\begin{equation*}
\Phi^W_\lambda(x;m,\lambda)=\frac{(x - m)^2 - \frac{2}{\lambda^2}}{2\lambda}, \qquad \Phi^W_m(x;m,\lambda)= x - m - \frac{1}{\lambda},
\end{equation*}
and the WIM satisfies 
\begin{equation*}
G_W(m,\lambda)=\begin{pmatrix}
    1 & \frac{1}{\lambda^2}     \\
    \frac{1}{\lambda^2} & \frac{2}{\lambda^4}
\end{pmatrix}.
\end{equation*}
Similarly to Gaussian cases, Wasserstein sufficient statistics are also those which can be written as quadratic functions of variables $x$.
\par
While the counterpart for Fisher case reads:
\begin{equation*}
\begin{split}
\Phi_\lambda^F(x;m,\lambda) = m - x + \frac{1}{\lambda}, \quad \Phi_m^F(x;m,\lambda) \text{ not well defined}.
\end{split}
\end{equation*}
Meanwhile, the Fisher information matrix is also ill-behaved, in contrast to the well-definedness of both Wasserstein score functions and WIMs. This example provides a situation where Wasserstein statistics are better than the classical Fisher statistics.
\end{example}

\section{Wasserstein natural gradient works efficiently}\label{Efficiency}
\par
In this section, we study Wasserstein dynamics in terms of sampling and estimation processes. As a consequence, we prove asymptotic efficiencies of the Wasserstein natural gradient algorithm. And we refer it as Wasserstein efficiency. Meanwhile, another efficiency that we named Poincar{\'e} efficiency is introduced and connected to Poincar{\'e} inequalities and log-Sobolev inequalities, which are discussed in section \ref{sec-Func}.

\par
In the beginning, we review the natural gradient algorithm. We aim to estimate an un-known distribution in a probability family $p(x;\theta)$ with unknown parameters $ \theta \in \Theta$. Assume that an optimal parameter $p(x;\theta_*)$ coincides with the target distribution. Given a set of i.i.d. samples $x_i,i = 1,2,...$ from this distribution, we utilize a general online natural gradient algorithm to solve this problem:
\begin{equation}\label{iterate}
	\theta_{t+1} = \theta_t -\frac{1}{t}\nabla_{\theta}^W l(x_t, \theta_t).
\end{equation}
In the above formula, $\theta_t$ is an updating state variable, $\frac{1}{t}$ in the RHS is an adaptive factor. And $\nabla_{\theta}^W$ is the Riemannian (natural) gradient of the loss function $l$ w.r.t. $\theta$ in Wasserstein metric. It can also be understood as using WIM as a preconditioner to get a new gradient direction, i.e. $\nabla^W_\theta l=G_W^{-1} \nabla_{\theta}l,$ with $\nabla_{\theta} l$ being the Euclidean gradient. We first pose here a definition of the efficiency of the natural gradient algorithm, which generalizes the notion discussed in \cite{NG}. Denote the Wasserstein covariance matrix of estimator $\theta_t$ by:
\begin{equation*}
    V_{t} = \EE_{p_{\theta_*}}\lp \nabla_x (\theta_t - \theta_* ) \cdot \nabla_x (\theta_t - \theta_* )^T \rpz,
\end{equation*}
where $\nabla_x (\theta_t - \theta_* )$ is the matrix given by
\begin{equation*}
    \nabla_x (\theta_t - \theta_* ) = \begin{pmatrix}
        \frac{\pa (\theta_t - \theta_* )_1 }{\pa x_1} & \frac{\pa (\theta_t - \theta_* )_1 }{\pa x_2} & \cdots & \frac{\pa (\theta_t - \theta_* )_1 }{\pa x_n}     \\
        \frac{\pa (\theta_t - \theta_* )_2 }{\pa x_1} & \frac{\pa (\theta_t - \theta_* )_2 }{\pa x_2} & \cdots & \frac{\pa (\theta_t - \theta_* )_2 }{\pa x_n}     \\
        \vdots & \vdots & \ddots & \vdots \\
        \frac{\pa (\theta_t - \theta_* )_n }{\pa x_1} & \frac{\pa (\theta_t - \theta_* )_n }{\pa x_2} & \cdots & \frac{\pa (\theta_t - \theta_* )_n }{\pa x_n}
    \end{pmatrix},
\end{equation*}
and the multiplication is simply in matrix sense. It turns out that the element of the covariance matrix is given by
\begin{equation*}
    \begin{aligned}
        \EE_{p_{\theta_*}}\lp \nabla_x (\theta_t - \theta_* ) \cdot \nabla_x (\theta_t - \theta_* )^T \rpz_{ij} = & \ \EE_{p_{\theta_*}}\lp \nabla_x (\theta_t - \theta_* )_i \cdot \nabla_x (\theta_t - \theta_* )_j \rpz
    \end{aligned}
\end{equation*}
Here $\cdot$ refers to inner product of gradient vectors. And subscripts $p(\cdot;\theta_*)$ refer to take expectation on the set of samples $x_i \sim p(\cdot;\theta_*), i = 1,2,...$. Notice that in this algorithm, we obtain the $t$-th estimator $\theta_t$ via $t-1$ iterations of the above equation \eqref{iterate}. Then we actually have $\theta_t = \theta_t(x_1,x_2...,x_{t-1})$. Hence intuitively, the Wasserstein-Cramer-Rao bound for $\theta_t$ is given by $\frac{1}{t - 1}G_W^{-1}\lp \theta_* \rpz$. It inspires the following definition:
\begin{definition}
    The Wasserstein natural gradient is asymptotic efficient if
    \begin{equation*}
        V_t = \frac{1}{t}G_W^{-1}\lp \theta_* \rpz + O(\frac{1}{t^2}).
    \end{equation*}
\end{definition}
\begin{remark}
    This definition is similar to the definition of classical Fisher efficiency except that we substitute the Fisher information matrix by the WIM. This also indicates the importance of studying information matrices. And it will be shown that this quantity characterizes convergence rates of dynamics in statistical inference problems.
\end{remark}
We first state a general updating equation for this dynamics. Then, we specify two different loss functions, namely, Fisher scores and Wasserstein scores. And we discuss convergence properties of these two cases separately.
\begin{theorem}[Variance updating equation of the Wasserstein natural gradient]\label{update}
For any function $l(x,\theta)$ that satisfies the condition $\EE_{p_\theta}l(x,\theta) = 0$, consider here the asymptotic behavior of the Wasserstein dynamics $\theta_{t+1} = \theta_t - \frac{1}{t}G_W^{-1}(\theta_t)l(x_t,\theta_t)$. That is, assume priorly $\mbE_{p_{\theta_*}} \lb \lp \theta_t - \theta_* \rpz^2 \rb$, \\ $\mbE_{p_{\theta_*}} \lb \lv \nabla_x \lp \theta_t - \theta_* \rpz \rv^2 \rb = o(1), \ \forall t$. Then, the Wasserstein covariance matrix $V_t$ updates according to the following equation: 
\begin{equation*}
    \begin{aligned}
     V_{t+1} = & \ V_t + \frac{1}{t^2} G_W^{-1}(\theta_*) \mbE_{p_{\theta_*}} \lb \nabla_x \lp l(x_t,\theta_*) \rpz \cdot \nabla_x \lp l(x_t,\theta_*)^{T} \rpz  \rb \left( G_W^{-1}(\theta_*) \right) \\ 
     & - \frac{2V_t}{t} \mbE_{p_{\theta_*}} \lb \nabla_{\theta}l(x_t,\theta_*) \rb G_W^{-1} (\theta_*) + o(\frac{V_t}{t}) + o\left( \frac{1}{t^2} \right).
	\end{aligned}
\end{equation*}
\end{theorem}
\begin{remark}
    In general, it will be shown that such a simple updating equation merely attributes to properties of information matrices. Specifically, any statistical information matrices with separability property w.r.t. independent variables have this form of updating equation. For the WIM, this is already established in proposition \ref{sep}. And for the Fisher information matrix, this is only a property of expectation under independent variables. Further results such as efficiency of the natural gradient can be established with the same procedure below.
\end{remark}

\par
The proof is technical and we leave it to section \ref{proof-main}. 
Here, we show several important cases of Theorem \ref{update}.
\subsection{Wasserstein natural gradient for Wasserstein scores}
\par
In Fisher case studied by \cite{NG}, we have:
\begin{equation*}
    \nabla_{\theta}^F l(x_t, \theta_t) = G_F^{-1}\Phi^F\lp x_t; \theta_t \rpz,
\end{equation*}
with $l(x_t, \theta_t)$ the log-likelihood function. Thus a natural generalization to Wasserstein geometry is: 
\begin{equation}
    \nabla_{\theta}^W l(x_t,\theta_t) = G_W^{-1} \Phi^W\lp x_t; \theta_t \rpz.
\end{equation}
\indent
 Concerned with this dynamics, we have the following corrolary.
\begin{corollary}[Wasserstein natural gradient efficiency]\label{WNE}
    For the dynamics
\begin{equation*}
    \theta_{t+1} = \theta_t - \frac{1}{t}G_W^{-1}(\theta_t)\Phi^W(x_t; \theta_t),
\end{equation*}
the Wasserstein covariance updates according to
\begin{equation*}
    \begin{aligned}
     V_{t+1} = & V_t + \frac{1}{t^2} G_W^{-1}(\theta_*) - \frac{2}{t} V_t + o\left( \frac{1}{t^2} \right) + o(\frac{V_t}{t}).
	\end{aligned}
\end{equation*}
Then, the online Wasserstein natural gradient algorithm is Wasserstein efficient, that is:
\begin{equation}\label{Cramer-Rao}
    V_t = \frac{1}{t}G_W^{-1}\lp \theta_* \rpz + O\lp \frac{1}{t^2} \rpz.
\end{equation}
\end{corollary}
\begin{proof}[Proof of Corollary \ref{WNE}]
    If we choose function $f(x,t)$ to be Wasserstein scores $\Phi^W_i$, we will have following simplification:
    \begin{equation*}
        \mbE_{p_{\theta_*}} \lb \nabla_x \lp \Phi^W(x_t; \theta_*) \rpz \cdot \nabla_x \lp \Phi^W(x_t; \theta_*)^{T} \rpz  \rb = G_W(\theta_*),
    \end{equation*}
    since $\Phi^W$ is the dual coordinate of the statistical model. We also have:
    \begin{equation*}
        \mbE_{p_{\theta_*}} \lb \nabla_{\theta}\Phi^W(x_t; \theta_*) \rb = - G_W(\theta_*),
    \end{equation*}
    which is given by differentiating $\mbE_{p_{\theta_*}} \lb \Phi^W(x_t; \theta_*) \rb = \mathbf{0}$ by $\theta$:
    \begin{equation*}
        \begin{aligned}
            \mathbf{0} = & \ \nabla_{\theta}\mbE_{p_{\theta_*}} \lb \Phi^W(x_t; \theta_*) \rb        \\
            = & \ \nabla_{\theta} \lb \int_{\cx} p(x;\theta_*)\Phi^W(x; \theta_*)dx\rb      \\
            = & \ \int_{\cx} \nabla_{\theta}p(x;\theta_*)\Phi^W(x; \theta_*)dx + \int_{\cx} p(x;\theta_*)\nabla_{\theta}\Phi^W(x; \theta_*)dx      \\
            = & \ G_W + \int_{\cx} p(x;\theta_*)\nabla_{\theta}\Phi^W(x; \theta_*)dx,
        \end{aligned}
    \end{equation*}
    where the last equality holds because of the pairing between tangent vector and cotangent vector. And the final updating equation for the Wasserstein covariance reduces to:
    \begin{equation*}
        V_{t+1} = V_t + \frac{1}{t^2} G_W^{-1}(\theta_*) - \frac{2}{t} V_t + O\left( \frac{1}{t^3} \right) + o(\frac{V_t}{t}).
    \end{equation*}
    \indent
    To further solve this updating equation, we expand $V_t = \frac{x}{t} + \frac{y}{t^2} + o\lp \frac{1}{t^2} \rpz$ with constant $x,y$ to be determined and plug into the equation(we ignore the term that is of order $o\lp \frac{1}{t^2} \rpz$):
    \begin{equation*}
        \frac{x}{t + 1} + \frac{y}{\lp t + 1 \rpz^2} + o\lp \frac{1}{t^2} \rpz = \frac{x}{t} + \frac{y}{t^2} + o\lp \frac{1}{t^2} \rpz + \frac{1}{t^2}G_W^{-1}\lp\theta_*\rpz - \frac{2x}{t^2} + o\lp \frac{1}{t^2} \rpz,
    \end{equation*}
    which is equivalent to:
    \begin{equation*}
        \lp \frac{x}{t + 1} - \frac{x}{t} \rpz + \lp \frac{y}{\lp t + 1 \rpz^2} - \frac{y}{t^2} \rpz + \frac{2x}{t^2} - \frac{1}{t^2}G_W^{-1}\lp\theta_*\rpz + o\lp \frac{1}{t^2} \rpz = 0.
    \end{equation*}
    And we conclude that:
    \begin{equation*}
        x = G_W^{-1}\lp\theta_*\rpz.
    \end{equation*}
    Thus, we asymptotically have following estimation on the Wasserstein covariance concerned with this dynamics:
    \begin{equation*}
    V_t = \frac{1}{t}G_W^{-1}\lp \theta_* \rpz + o\lp \frac{1}{t} \rpz.
    \end{equation*}
\end{proof}
\begin{remark}
    At first, such a generalization to Wasserstein metric may seem unreasonable. We only use a fact that both of them are metrics on probability spaces. Different from Fisher scores $\Phi^F = \nabla_{\theta}l\lp x; \theta \rpz$, Wasserstein scores $\Phi^W$ can not be written as gradients of some functions w.r.t. $\theta$. There is no such ``loss functions''. However, a key insight here is that, if in a second we assume that the statistical model $\Theta$ is exactly the density manifold $G_W(p_\theta) = g_W(p_\theta), G_F(p_\theta) = g_F(p_\theta)$: 
\begin{equation*}
    \begin{aligned}
        G_W^{-1}(p_\theta) \Phi^W\lp x; \theta \rpz = & \  g_W(p_\theta)g_W^{-1}(p_\theta)\frac{\pa}{\pa \theta}p(x;\theta) = \nabla_{\theta} p(x_t;\theta_t)    \\
        = & \  g_F(p_\theta)g_F^{-1}(p_\theta)\frac{\pa}{\pa \theta}p(x;\theta) = G_F^{-1}(p_\theta) \Phi^F\lp x; \theta \rpz.
    \end{aligned}
\end{equation*}
Then both two dynamics can be written in the following way:
\begin{equation*}
    \theta_{t+1} = \theta_t -\frac{1}{t}\nabla_{\theta}p(x_t; \theta_t).
\end{equation*}
\end{remark}

\subsection{Wasserstein natural gradient for Fisher scores}
\par
Another phenomenon appears when we consider the Wasserstein natural gradient applies to Fisher scores. Specifically, we use log-likelihood function as a loss function and apply WIM as a preconditioner. The dynamics concerned in this case is given by:
\begin{equation*}
    \theta_{t+1} = \theta_t - \frac{1}{t}\nabla_{\theta}^W l(x_t,\theta_t).
\end{equation*}
The Wasserstein natural gradient is simply $\nabla_{\theta}^W l\lp x_t,\theta_t \rpz = G_W^{-1}\nabla_{\theta} l \lp x_t,\theta_t \rpz$. We comment here that $\nabla_{\theta}l \lp x_t,\theta_t \rpz = \Phi^F(x_t,\theta_t)$ is both the Euclidean gradient of log-likelihood function $l$ w.r.t. $\theta$ and the Fisher score. And the convergence analysis is shown in the following corollary:
\begin{corollary}[Poincar\'e efficiency]\label{WFNE}
For the dynamics
\begin{equation*}
    \theta_{t+1} = \theta_t - \frac{1}{t}\nabla_{\theta}^W l(x_t,\theta_t),
\end{equation*}
the Wasserstein covariance updates according to
\begin{equation*}
	\begin{aligned}
	V_{t+1} = & \ V_t + \frac{1}{t^2} G_W^{-1}(\theta_*) \mbE_{p_{\theta_*}} \lb \nabla_x \lp \nabla_{\theta}l(x_t,\theta_*) \rpz \cdot \nabla_x \lp \nabla_{\theta}l(x_t,\theta_*)^{T} \rpz  \rb G_W^{-1}(\theta_*) \\
& \ - \frac{2}{t} V_t G_F (\theta_*) G_W^{-1} (\theta_*) + O\left( \frac{1}{t^3} \right) + o\lp \frac{V_t}{t} \rpz .
	\end{aligned}
\end{equation*}
Now suppose that $\alpha = sup \{ a | G_F \succeq a G_W \}$. Then the dynamics is characterized by the following formula
\begin{equation}\label{Poincare}
    \begin{aligned}
        V_t & = \lbb 
        \begin{aligned}
            & \ O\lp t^{-2\alpha} \rpz,  \qquad \qquad \qquad \qquad \qquad \qquad \qquad \qquad \qquad \  2\alpha \leq 1,   \\
            & \  \frac{1}{t}\lp 2G_FG_W^{-1} - \I \rpz^{-1}{G_W^{-1}(\theta_*) \mathfrak{I} \left( G_W^{-1}(\theta_*) \right)} + O\lp \frac{1}{t^2} \rpz, \quad 2\alpha > 1,
        \end{aligned}\right.
        \end{aligned}
        \end{equation}
where \begin{equation*}
        \mathfrak{I} = \mbE_{p_{\theta_*}} \lb \nabla_x \lp \nabla_{\theta}l(x_t,\theta_*) \rpz \cdot \nabla_x \lp \nabla_{\theta}l(x_t,\theta_*)^{T} \rpz  \rb,
\end{equation*}
where elements of this matrix is given by
\bequn
    \mathfrak{I}_{ij} = \mbE_{p_{\theta_*}} \lb \nabla_x \lp \nabla_{\theta_i}l(x_t,\theta_*) \rpz \cdot \nabla_x \lp \nabla_{\theta_j}l(x_t,\theta_*)^{T} \rpz  \rb
\eequn

\end{corollary}
\begin{proof}[Proof of Corollary \ref{WFNE}]
    The result is obtained once we observe that:
    \begin{equation*}
        \mbE_{p_{\theta_*}} \lb \nabla_{\theta}\Phi^F(x_t,\theta_*) \rb = - G_F(\theta_*),
    \end{equation*}
    which follows exactly the same philosophy of the previous case. We conclude that the Wasserstein covariance updates according to:
    \begin{equation*}
	\begin{aligned}
	V_{t+1} = & \ V_t + \frac{1}{t^2} G_W^{-1}(\theta_*) \mbE_{p_{\theta_*}} \lb \nabla_x \lp \nabla_{\theta}l(x_t,\theta_*) \rpz \cdot \nabla_x \lp \nabla_{\theta}l(x_t,\theta_*)^{T} \rpz  \rb \left( G_W^{-1}(\theta_*) \right) \\
& \  - \frac{2}{t} V_t G_F (\theta_*) G_W^{-1} (\theta_*) + O\left( \frac{1}{t^3} \right) + o\lp \frac{V_t}{t} \rpz .
	\end{aligned}
\end{equation*}
    \par
    Next, we solve this dynamics asymptotically. We denote $G_F (\theta_*) G_W^{-1} (\theta_*) = B$ and $G_W^{-1}(\theta_*) \mbE_{p_{\theta_*}} \lb \nabla_x \lp \nabla_{\theta}l(x_t,\theta_*) \rpz \cdot \nabla_x \lp \nabla_{\theta}l(x_t,\theta_*)^{T} \rpz  \rb \left( G_W^{-1}(\theta_*) \right) = C$.
    \par
    Now by elementary linear algebra, we know that the matrix $B = G_F (\theta_*) G_W^{-1} (\theta_*)$ is similar to the matrix $G_W^{-\half}G_FG_W^{-\half}$. Hence their eigenvalues coincide. While the definition of $\alpha$ translates to that the least eigenvalue of the symmetric matrix $G_W^{-\half}G_FG_W^{-\half}$ is exactly $\alpha$. Thus we conclude that the least eigenvalues of the matrix $B$ are also $\alpha$. Suppose first that $2\alpha < 1$, we consider the following expansion of matrix $V_t$:
    \begin{equation*}
        V_t = \frac{A_1}{t^q} + \frac{A_2}{t^{q + 1}} + o\lp \frac{1}{t^{q+1}} \rpz, \qquad A_1, A_2 = O(1).
    \end{equation*}
    And plug the above equation to both sides of the updating equation, we find:
    \begin{equation*}
        \frac{A_1}{\lp t + 1 \rpz^q} + \frac{A_2}{\lp t + 1 \rpz^{q + 1}} + o\lp \frac{1}{t^{q+1}} \rpz = \frac{A_1}{t^q} + \frac{A_2}{t^{q + 1}} + o\lp \frac{1}{t^{q+1}} \rpz -  \frac{2A_1B}{t^{q + 1}} + \frac{C}{t^2} + O(\frac{1}{t^3}).
    \end{equation*}
    Using the Lagrange's mean value theorem, we have:
    \begin{equation*}
        \frac{A}{t^q} - \frac{A}{\lp t + 1 \rpz ^q} = \frac{qA}{\lp t + \upsilon \rpz^{q + 1}} = \frac{qA}{t^{q + 1}} + o(\frac{1}{t^q}), \qquad \upsilon \in [0,1].
    \end{equation*}
    Substituting back to the equation, we get:
    \begin{equation*}
        \mathbf{0} = \frac{A_1 \lp q\I - 2B \rpz }{t^{q + 1}} + \frac{C}{t^2} + o(\frac{1}{t^{q+1}}) + O(\frac{1}{t^3}).
    \end{equation*}
    We conclude that we cannot have $q$ strictly greater than $1$, for then the most significant term in the RHS will be $\frac{C}{t^2} \neq \mathbf{0}$ which contradicts to the LHS. Thus if we have $q < 2\alpha < 1$, the matrix $q\I - 2B$ will be negative definite, and we cannot have $A_1\lp q\I - 2B \rpz = \mathbf{0}$ unless $A_1$ equals to $0$. Consequently, the index $q$ should be greater than or equal to $2\alpha$. And we have that asymptotically
    \begin{equation*}
        V_t = O\lp \frac{1}{t^{2\alpha}}\rpz.
    \end{equation*}
    \par
    While for the situation such that $2\alpha > 1$, we expand $V_t = \frac{A_1}{t} + \frac{A_2}{t^2} + o\lp \frac{1}{t^2} \rpz$ with constant $A_1,A_2$ to be determined:
    \begin{equation*}
        \frac{A_1}{t + 1} + \frac{A_2}{\lp t + 1 \rpz^2} + o\lp \frac{1}{t^2} \rpz = \frac{A_1}{t} + \frac{A_2}{t^2} + o\lp \frac{1}{t^2} \rpz + \frac{C}{t^2} - \frac{2A_1B}{t^2} + o\lp \frac{1}{t^2} \rpz.
    \end{equation*}
    The constant $A_1$ can be fixed by considering the coefficient of the term $\frac{1}{t^2}$ for both sides with conclusion:
    \begin{equation*}
        A_1 = \lp 2B - \I \rpz^{-1}C.
    \end{equation*}
    Here, invertibility of the matrix $2B - \I$ is guaranteed by the fact that eigenvalues of $2B$ are all greater than $1$, thus the matrix $2B - \I$ is indeed positive definite.
\end{proof}
The convergence behavior of this dynamics relies largely on the least significant eigenvalue of the matrix $G_FG_W^{-1}$. This is in great similarity with the RIW condition for Poincar{\'e} inequality in statistical models \cite{LiMontufar2018_riccia}. This inspires us to name such efficiency Poincar{\'e} efficiency. For detailed discussions and calculations on Poincar\'e inequalities in statistical models, please refer to the section \ref{sec-Func}. We also illustrate some results of two efficiencies in Gaussian family, whose proof and numerical experiments are delayed to the section \ref{num-exp}. 
\begin{example}[Gaussian distribution]
Suppose we have following dynamics in a Gaussian model $p\lp x; \mu, \sigma \rpz$
\begin{equation*}
    \theta_{t+1} = \theta_t - \frac{1}{t}\nabla_{\theta}^W l(x_t,\theta_t), \quad x_t \sim p\lp x; \mu_*, \sigma_* \rpz.
\end{equation*}
The asymptotic behavior of the covariance matrix for the online Wasserstein natural gradient algorithm is given by
\begin{equation*}
       V_t = \left\{
    \begin{aligned}
        & O\lp t^{-\frac{2}{\sigma_*^2}} \rpz, \qquad \qquad \qquad \qquad \qquad \quad  \ \ \frac{1}{\sigma_*^2} \leq \half,        \\
        & \frac{1}{t}\begin{pmatrix}
        \frac{1}{\lp 2 - \sigma_*^2 \rpz \sigma_*^2} & 0  \\
        0 & \frac{4}{\lp 4 - \sigma_*^2 \rpz \sigma_*^2}
    \end{pmatrix} + O(\frac{1}{t^2}), \qquad \frac{1}{\sigma_*^2} > \half.
    \end{aligned}\right. 
\end{equation*}
\end{example}

\section{Discussions}
In this paper, we introduce the Wasserstein information matrix in statistical models. Similar to the study in information geometry, we turn the geometric aspect of the Wasserstein metric into statistics. Here we generalize the classical concepts such as score function, covariance operator, Cramer-Rao bound, and estimation to the Wasserstein statistics. Several explicit computable examples are provided, including the location-scale family, and the ReLU push-forward family. Also, by comparing both Wasserstein and Fisher information matrices, some new efficiency concepts, such as Wasserstein efficiency and Poincar{\'e} efficiency have been introduced. 

In the future, several natural questions between Fisher and Wasserstein statistics arise. For example, similar to the relation with Fisher information matrices and maximal likelihood estimators, what is the relation between the WIM and the Wasserstein distance estimator? Is there a canonical Wasserstein divergence function for the WIM? What is the corresponding Wasserstein maximal likelihood estimator? Meanwhile, we will apply Wasserstein natural gradient to study stochastic gradient descent algorithms in statistical learning problems. Lastly and most importantly, we have shown that the Wasserstein statistics provide the rigorous statistical advantages in generative models than classical Fisher statistics. We will study the properties of WIMs in clear statistical terms for machine learning models. 

\bibliographystyle{abbrv}
\bibliography{WIM_arxiv}
\newpage

\appendix
\section{Proofs in section \ref{sec2}}\label{supp-2}
\subsection{WIMs and score functions in analytic examples}\label{ana-exp}
\begin{proof}[Proof of WIMs in Gaussian families]
Since we have
\begin{equation*}
\log p(x;\mu,\sigma)=-\frac{(x-\mu)^2}{2\sigma^2}-\log\sigma-\log\sqrt{2\pi},
\end{equation*}
taking derivative, we get
\begin{equation*}
    \begin{aligned}
\nabla_x\log p(x;\mu,\sigma) & =-\frac{x-\mu}{\sigma^2},\\
\frac{\partial}{\partial\mu}\log p(x;\mu,\sigma) & =\frac{x-\mu}{\sigma^2},\quad  \frac{\partial}{\partial\sigma}\log p(x;\mu,\sigma)=\frac{(x-\mu)^2}{\sigma^3}-\frac{1}{\sigma}.
    \end{aligned}
\end{equation*}
In this case, the Possion equation for Wasserstein score functions $(\Phi^W_\mu, \Phi^W_\sigma)$ forms 
\begin{equation*}
\left\{\begin{aligned}
& -\frac{x-\mu}{\sigma^2}\cdot \frac{d}{dx}\Phi^W_\mu+\frac{d^2}{dx^2}\Phi^W_\mu=-\frac{x-\mu}{\sigma^2}, \\
& -\frac{x-\mu}{\sigma^2}\cdot \frac{d}{dx}\Phi^W_\sigma+\frac{d^2}{dx^2}\Phi^W_\sigma=-\frac{(x-\mu)^2}{\sigma^3}+\frac{1}{\sigma}.
\end{aligned}\right.
\end{equation*}
We simply check that $\Phi^W_\mu(x) = x - \mu$ and $\Phi^W_\sigma(x)=\frac{(x-\mu)^2 - \sigma^2}{2\sigma}$ are solutions, and they also satisfy the normalization condition $\EE_{p_\theta} \Phi_i^W = 0$. Thus 
\begin{equation*}
\begin{aligned}
G_W(\mu, \sigma)_{\mu\mu}=& \ \mathbb{E}_{p_{\mu,\sigma}} \lp \frac{d}{dx}\Phi^W_\mu, \frac{d}{dx}\Phi^W_\mu \rpz =\mathbb{E}_{p_{\mu,\sigma}} 1 =1,\\
G_W(\mu, \sigma)_{\mu\sigma}=& \ \mathbb{E}_{p_{\mu,\sigma}} \lp \frac{d}{dx}\Phi^W_\mu, \frac{d}{dx}\Phi^W_\sigma \rpz = \mathbb{E}_{p_{\mu,\sigma}} \lp 1\cdot (-\frac{X-\mu}{2\sigma})\rpz = 0,        \\
G_W(\mu, \sigma)_{\sigma\sigma}=& \ \mathbb{E}_{p_{\mu,\sigma}} \lp \frac{d}{dx}\Phi^W_\sigma, \frac{d}{dx}\Phi^W_\sigma \rpz = \mathbb{E}_{p_{\mu,\sigma}} \lp \frac{X-\mu}{\sigma}\cdot \frac{X-\mu}{\sigma} \rpz = 1.
\end{aligned}
\end{equation*}
\end{proof}
\begin{proof}[Proof of WIMs in exponential families]
We derive results using the closed-form solution in 1-d. The cumulative distribution function satisfies 
\begin{equation*}
F(x;m,\lambda)=\begin{cases}
1-e^{-\lambda (x - m)}& x\geq m,\\
0 & x<m.
\end{cases}
\end{equation*}
Thus 
\begin{equation*}
    \begin{aligned}
\frac{\partial}{\partial \lambda}F(x;m,\lambda)=\begin{cases}
 (x - m) e^{-\lambda (x - m)}&\textrm{$x\geq m$,}\\
0 & x < m.
\end{cases}\\ 
\frac{\partial}{\partial m}F(x;m,\lambda)=\begin{cases}
 \lambda e^{-\lambda (x - m)}&\textrm{$x\geq m$,}\\
0 & x < m.
\end{cases}
    \end{aligned}
\end{equation*}
Then 
\begin{equation*}  
    \begin{aligned}
        \Phi^W_\lambda(x;m,\lambda) = & -\int_m^{x}\frac{1}{p(y;m,\lambda)}\frac{\partial}{\partial\lambda}F(y;m,\lambda)dy + C_1   \\
        = & -\int_m^x\frac{(y - m)}{\lambda}dy + C_1 = \frac{(x - m)^2}{2\lambda} + C_1,          \\
        \Phi^W_m(x;m,\lambda) = & -\int_m^{x}\frac{1}{p(y;m,\lambda)}\frac{\partial}{\partial m}F(y;m,\lambda)dy + C_2 \\ 
        = & -\int_m^x dy + C_2 = (x - m) + C_2.          \\
    \end{aligned}
\end{equation*}
Using the normalization condition, we can decide integration constants appearing above. And inner products between score functions follow:
\begin{equation*}
\begin{split}
G_W(m,\lambda)_{\lambda\lambda} = & \  \mathbb{E}_{p_{m,\lambda}} \lp \frac{d}{dx}\Phi^W_\lambda, \frac{d}{dx}\Phi^W_\lambda \rpz     \\
= & \int_m^{\infty} \frac{(x - m)}{\lambda}\cdot  \frac{(x - m)}{\lambda}\cdot  \lambda e^{-\lambda \lp x - m \rpz}dx \\
= & \ \int_m^{\infty} \frac{(x - m)^2}{\lambda}e^{-\lambda \lp x - m \rpz}dx = \frac{2}{\lambda^4},           \\
G_W(m,\lambda)_{\lambda m} = & \  \mathbb{E}_{p_{m,\lambda}} \lp \frac{d}{dx}\Phi^W_\lambda, \frac{d}{dx}\Phi^W_m \rpz = \int_m^{\infty} \frac{(x - m)}{\lambda}\cdot  \lambda e^{-\lambda \lp x - m \rpz}dx = \frac{1}{\lambda^2},     \\
G_W(m,\lambda)_{mm} = & \ \mathbb{E}_{p_{m,\lambda}} \lp \frac{d}{dx}\Phi^W_m, \frac{d}{dx}\Phi^W_m \rpz = \int_m^{\infty} \lambda e^{-\lambda \lp x - m \rpz}dx = 1.
\end{split}
\end{equation*}
\end{proof}
\begin{proof}[Proof of WIMs in uniform families]
The cumulative distribution function satisfies 
\begin{equation*}
F(x;a,b)=\begin{cases}
1& x> b,\\
\frac{x-a}{b-a} & a\leq x \leq b,\\
0 & x<a.
\end{cases}
\end{equation*}
Thus when $x\in [a,b]$, 
\begin{equation*}
\frac{\partial}{\partial a}F(x;a,b)=\frac{x-b}{(b-a)^2},\quad \frac{\partial}{\partial b}F(x;a,b)=\frac{a-x}{(b-a)^2}.
\end{equation*}
Then 
\begin{equation*}
\begin{aligned}
\Phi^W_a(x;a,b)=&-\int_a^{x}\frac{1}{p(y;a,b)}\frac{\partial}{\partial a}F(y;a,b)dy + C_1 = \frac{(x-a)(a-2b+x)}{2(b-a)} + C_1,\\
\Phi^W_b(x;a,b) = &-\int_a^{x}\frac{1}{p(y;a,b)}\frac{\partial}{\partial b}F(y;a,b)dy + C_2 = \frac{(a-x)^2}{2(b-a)} + C_2,
\end{aligned}
\end{equation*}
where integration constants $C_1,C_2$ can be decided via the normalization condition. Thus
\begin{equation*}
\begin{aligned}
G_W(a,b)_{aa} = & \ \mathbb{E}_{p_{a, b}} \lp \frac{d}{dx}\Phi^W_a, \frac{d}{dx}\Phi^W_a \rpz = \frac{1}{3},\\
G_W(a,b)_{ab} = & \ \mathbb{E}_{p_{a, b}} \lp \frac{d}{dx}\Phi^W_a, \frac{d}{dx}\Phi^W_b \rpz = \frac{1}{6},\\
G_W(a,b)_{bb} = & \ \mathbb{E}_{p_{a, b}} \lp \frac{d}{dx}\Phi^W_b, \frac{d}{dx}\Phi^W_b \rpz = \frac{1}{3}.
\end{aligned}
\end{equation*}
\end{proof}
\begin{proof}[Proof of the WIM in semicircle families]
The cumulative distribution function satisfies 
\begin{equation*}
\begin{split}
F(x + m;m,R)=& \ \int_{-R}^x \frac{2}{\pi R^2}\sqrt{R^2-y^2}dy\\
=& \ \int_{-\frac{\pi}{2}}^{\arcsin(\frac{x}{R})}\frac{2}{\pi R^2}\sqrt{R^2-R^2\sin^2t}d (R\sin t)\\
=& \ \int_{-\frac{\pi}{2}}^{\arcsin(\frac{x}{R})}\frac{2}{\pi R^2} R^2(\cos t)^2 dt\\
=& \ \int_{-\frac{\pi}{2}}^{\arcsin(\frac{x}{R})}\frac{1}{2\pi}\frac{\cos(2t)+1}{2} dt\\
=& \ \frac{1}{\pi}\Big(\frac{\sin (2t)}{2}+t\Big)\Bigg|_{-\frac{\pi}{2}}^{\arcsin\frac{x}{R}}\\
=& \ \frac{1}{\pi}\Big\{\frac{x\sqrt{R^2-x^2}}{R^2}+\arcsin\frac{x}{R}+\frac{\pi}{2}\Big\},
\end{split}
\end{equation*}
where we use a transformation $y=R\sin t$. Thus
\begin{equation*}
\begin{split}
\frac{\partial}{\partial R}F(x + m;m,R)= &\ \frac{1}{\pi}\Big\{\frac{(x\sqrt{R^2-x^2})'R^2-2Rx\sqrt{R^2-x^2}}{R^4}+(\arcsin\frac{x}{R})
'\Big\}\\
=& \ \frac{1}{\pi}\Big\{  \frac{xR(R^2-x^2)^{-\frac{1}{2}}R^2-2Rx\sqrt{R^2-x^2}}{R^4} -\frac{x}{R\sqrt{R^2-x^2}}\Big\}\\
=& \ -\frac{2x\sqrt{R^2-x^2}}{\pi R^3}.
\end{split}
\end{equation*}
Thus
\begin{equation*}
\begin{aligned}
\Phi^W_R(x + m;m,R)=& \ -\int_{-R}^{x}\frac{1}{p(y;m,R)}\frac{\partial}{\partial R}F(y;m,R)dy + C \\
=& \ \int_{-R}^x\frac{y}{R}dy + C \\
=& \ \frac{1}{R}(\frac{x^2}{2}-\frac{R^2}{2}) + C.
\end{aligned}
\end{equation*}
The calculation of the score function associated with the parameter $p$ is the same as before. And we conclude
\begin{equation*}
    \Phi^W_p(x;m,R)= x - m.
\end{equation*}
Thus
\begin{equation*}
\begin{aligned}
    G_W(m,R)_{mm} = & \ \mathbb{E}_{p_{m,R}} \lp \frac{d}{dx}\Phi^W_m, \frac{d}{dx}\Phi^W_m \rpz =  1,            \\
    G_W(m,R)_{mR} = & \ \mathbb{E}_{p_{m,R}} \lp \frac{d}{dx}\Phi^W_R, \frac{d}{dx}\Phi^W_m \rpz = \frac{1}{R^2}\mathbb{E}_{m_{R}} \lp x - m \rpz = 0,            \\
    G_W(m,R)_{RR} = & \ \mathbb{E}_{p_{m,R}} \lp \frac{d}{dx}\Phi^W_R, \frac{d}{dx}\Phi^W_R \rpz = \frac{1}{R^2}\mathbb{E}_{m_{R}} \lp x - m \rpz^2 = \frac{1}{4}.
\end{aligned}
\end{equation*}
\end{proof}
\subsection{The location-scale family}\label{location}
\begin{example}[Location-scale families]
Consider a location-scale family as following: given a probability density function $p(x)$ with $\int_{\RR} p(x)dx = 1$, we define density functions of a location-scale family with a location parameter $m$, and a scale parameter $\lambda$ as
\begin{equation*}
    p(x;m,\lambda) = \frac{1}{\lambda}p(\frac{x - m}{\lambda}),\qquad \lambda > 0. 
\end{equation*}
Most of previously discussed examples belong to this family, except that we do not use location and scale parameters in their parameterizations. We present some geometric formulas in this setting. We further require the original density function to be symmetric according to the location parameter $m$, i.e. $p(x) = p(2m - x)$. Notice that a simple corollary of this assumption is $\EE_{p_{m,\lambda}}x = m$.     \\
\indent
We use the closed-form solution for 1-d model to calculate the score function associated with the location parameter $m$. Thus we have:
\begin{equation*}
    \begin{aligned}
         \frac{\partial}{\partial m}F(x;m,\lambda) = & \frac{\partial}{\partial m} \int_{-\infty}^x p(y;m,\lambda)dy = \frac{\partial}{\partial m} \int_{-\infty}^x \frac{1}{\lambda}p(\frac{y - m}{\lambda})dy         \\
         = & - \frac{\partial}{\partial x} \int_{-\infty}^x \frac{1}{\lambda}p(\frac{y - m}{\lambda})dy = -p(x;m,\lambda).
    \end{aligned}
\end{equation*}
Consequently, the score function associated to the parameter $m$ satisfies
\begin{equation*}
    \Phi^W_m(x;m,\lambda)=-\int_m^{x}\frac{1}{p(y;m,\lambda)}\frac{\partial}{\partial m}F(y;m,\lambda)dy + C_1 = \lp x - m \rpz + C_1,
\end{equation*}
where the integration constant $C_1$ is determined to be 0. Thus we have
\begin{equation*}
	G_W(m,\lambda)_{mm} = \mathbb{E}_{p_{m,\lambda}} \lp \frac{d}{dx}\Phi^W_m, \frac{d}{dx}\Phi^W_m \rpz =  1.
\end{equation*} 
\indent
For the scaling parameter $\lambda$, we use a method of optimal transportation map to determine its score function. Namely, for two smooth distributions $p_1, p_2$ which are absolutely continuous w.r.t. each other, their Wasserstein distance can be obtained by an optimal transportation map $f$, i.e.
\begin{equation*}
    f_* p_1 = p_2, \quad W_2^2\lp p_1, p_2 \rpz = \int_{\cx} \lp f\lp x \rpz - x \rpz^2 p_1\lp x \rpz dx.
\end{equation*}
Assume we have a tangent vector $\frac{\partial p}{\partial \theta}$ and a smooth path $p\lp t \rpz \subset \cp\lp \cx \rpz, t\in\lb -\epsilon, \epsilon \rb$ with $p\lp 0 \rpz = p_0, p'\lp 0 \rpz = \frac{\partial p}{\partial \theta}$. Denote the optimal transportation map between $p\lp 0 \rpz, p\lp \theta \rpz$ as $f\lp x, \theta \rpz$. Then we have following relation between optimal transportation maps and the score function associated with tangent vector $\frac{\partial p}{\partial \theta}$
\begin{equation*}
    \frac{d}{dx}\Phi^W\lp x \rpz = \lim_{\Delta\theta \rightarrow 0}\frac{f\lp x, \Delta \theta \rpz - x}{\Delta \theta}.
\end{equation*}
\par
First, we show that the optimal transportation map between distributions $p(x;m_1,\lambda_1)$ and $p(x;m_2,\lambda_2)$ is given by a linear map:
\begin{equation*}
	l(x) = m_2+\frac{(x-m_1)\lambda_2}{\lambda_1}.
\end{equation*}
As we are working in a location-scale family, it is easy to show that this map pushes $p(x;m_1,\lambda_1)$ forward to $p(x;m_2,\lambda_2)$, i.e. $l_*p_{m_1,\lambda_1} = p_{m_2,\lambda_2}$. Then, we have
\begin{equation*}
      l(x) = \nabla_x \lp m_2\lp x - m_1 \rpz +  \frac{(x-m_1)^2\lambda_2}{2\lambda_1} \rpz. 
\end{equation*}
The function in the bracket is a convex function. Therefore, $l\lp x \rpz$ is exactly the optimal transportation map between these two distributions. 
\par
To calculate the score function correspondent to the tangent vector $\frac{\partial }{\partial \lambda}$, we consider following infinitesimal optimal transportation $p(x;m_1,\lambda_1) \rightarrow p(x;m_1,\lambda_1 + d\lambda)$. By discussions above, the optimal transportation map is given by
\begin{equation*}
    l\lp x \rpz = m_1+\frac{(x-m_1)\lp \lambda_1 + d\lambda \rpz}{\lambda_1} = x + \lp x - m_1 \rpz\frac{d\lambda}{\lambda_1}.
\end{equation*}
Thus the gradient of the score function is given by
\begin{equation*}
	\frac{d}{dx}\Phi^W_{\lambda}\lp x; m_1,\lambda_1 \rpz = \frac{l\lp x \rpz - x}{d\lambda} = \frac{(x-m_1)}{\lambda_1}.
\end{equation*}
The inner product of this tangent vector is given by
\begin{equation*}\label{ls-wasserstein}
    \begin{aligned}
        G_W(m,\lambda)_{\lambda \lambda} = & \mathbb{E}_{p_{m,\lambda}} \lp \frac{d}{dx}\Phi^W_{\lambda}, \frac{d}{dx}\Phi^W_{\lambda} \rpz     \\
        = & \int_{\mathbb{R}} \left(\frac{x - m}{\lambda}\right)^2 p(x;m,\lambda)dx \\
        = &\frac{\mathbb{E}_{p_{m,\lambda}}x^2-2m\mathbb{E}_{p_{m,\lambda}}x+m^2}{\lambda^2}.
    \end{aligned}
\end{equation*}
The gradient of the score function associated to the parameter $\lambda$ (resp. $m$) is odd (resp. even) function when viewing as a function of $x-m$. We conclude that the integration of their product is zero:
\begin{equation*}
	G_W(m,\lambda)_{\lambda m} = \mathbb{E}_{p_{m,\lambda}} \lp \frac{d}{dx}\Phi^W_{\lambda}, \frac{d}{dx}\Phi^W_m \rpz  = \EE_{p_{m,\lambda}}(x - m) = 0.
\end{equation*}
Consequently, WIMs of location-scale families are diagonal matrices, i.e.
\begin{equation*}
    G_W\lp m, \lambda \rpz = \begin{pmatrix}
        1 & 0     \\
        0 & \frac{\mathbb{E}_{p_{m,\lambda}}x^2-2m\mathbb{E}_{p_{m,\lambda}}x+m^2}{\lambda^2}
    \end{pmatrix}.
\end{equation*}

\par
We next explain above closed-form solutions of WIMs by following proposition.  
\begin{proposition}\label{total-geo}
    A location-scale family $p(x;m,\lambda)$ is a totally geodesic family in density manifold under Wasserstein metric.
\end{proposition}
\begin{proof}
    It suffices to prove that for any two densities $\rho_1 = p(x;m_1,\lambda_1)$ and $\rho_2 = p(x;m_2,\lambda_2)$, a geodesic connecting them lies within this family.
    We compute the optimal transport map $T$ associated with these two measures $\rho_1,\rho_2$, that is:
    \begin{equation*}
        T = \argmin_{T_* \rho_1 = \rho_2} \int_{\RR} \lp T(x) - x \rpz^2 \rho_1(x)dx,
    \end{equation*}
    where $T$ is a map that pushes density $\rho_1$ forward to density $\rho_2$. It is known that a sufficient and necessary condition for an optimal map in 1-d case is that it is a monotone map, i.e. $\lp T(x) - T(y) \rpz\lp x - y \rpz \geq 0$. And in a location-scale family, such map has a closed-form solution, namely:
    \begin{equation*}
        T(x) = \frac{\lambda_2\lp x - m_1 \rpz}{\lambda_1} + m_2.
    \end{equation*}
    \indent
   The geodesic $\gamma(t):[0,1] \rightarrow \cp(\RR)$ between $\rho_1$ and $\rho_2$ follows easily as below by the classical theory of optimal transport
    \begin{equation*}
        \gamma(t) = \lp tx + (1-t)T(x) \rpz_*\rho_1,
    \end{equation*}
    where the push-forward map has a closed-form solution:
    \begin{equation*}
        \begin{aligned}
        tx + (1-t)T(x) = & \ tx + (1 - t)\frac{\lambda_2\lp x - m_1 \rpz}{\lambda_1} + (1 - t)m_2 \\
        = & \  \frac{\lp t\lambda_1 + (1 - t)\lambda_2 \rpz \lp x - m_1 \rpz}{\lambda_1} + (1 - t)m_2 + tm_1.
        \end{aligned}
    \end{equation*}
    And by the same argument, $\gamma(t)$ lies in this location-scale family with parameters given by
    \begin{equation*}
        \lambda_t = t\lambda_1 + (1 - t)\lambda_2,\qquad m_t = (1 - t)m_2 + tm_1.
    \end{equation*}
    Thus we show that geodesics between any two densities in a location-scale family lie in this family. In other words, location-scale families are totally geodesic submanifolds in density manifold.
\end{proof}
\begin{remark}
    This result on totally geodesic of location-scale families is a generalization of the same result on Gaussian families in 1-d. Both proofs of these two cases rely on the fact that optimal transport maps in these families are linear.
\end{remark}
\begin{remark}
    \par
For location-scale families, we also formulate its Fisher scores and Fisher information matrices for comparisons:
\begin{equation}\label{ls-fisher}
    \begin{aligned}
        \Phi_{m}^F(x;m,\lambda) & = \frac{p'}{\lambda p},\quad
        \Phi_{\lambda}^F(x;m,\lambda) = -\frac{1}{\lambda} - \frac{(x - m)p'}{\lambda^2 p},      \\
    	G_F(m,\lambda)_{\lambda \lambda} & = \int_{\mathbb{R}} p \lp \pa_{\lambda} \log p \rpz^2 dx    = \int_{\mathbb{R}} p \lp -\frac{1}{\lambda} - \frac{(x - m)p'}{\lambda^2 p}\rpz^2 dx \\
    	& =  \frac{1}{\lambda^2} \lp 1 + \int_{\RR} \lp \frac{\lp x - m \rpz^2 p'^2}{\lambda^2 p} + \frac{\lp x - m \rpz p'}{\lambda}\rpz dx  \rpz,  \\
    	G_F(m,\lambda)_{mm} & =  \int_{\mathbb{R}} p \lp \pa_{m} \log p \rpz^2 dx = \frac{1}{\lambda^2}\int_{\mathbb{R}} \frac{p'^2}{p} dx,     \\
    	G_F(m,\lambda)_{m\lambda} & =  \int_{\mathbb{R}} p \lp \pa_{m} \log p \rpz \lp \pa_{\lambda} \log p \rpz dx \\
    	& = \int_{\mathbb{R}} p \lp -\frac{p'}{\lambda p} \rpz \lp -\frac{1}{\lambda} - \frac{(x - m)p'}{\lambda^2p} \rpz dx   \\
    	& = \int_{\mathbb{R}}  \frac{(x - m)p'^2}{\lambda^3p} dx.
    \end{aligned}
\end{equation}
WIMs and Fisher information matrices are given by
\begin{equation*}   
    \begin{aligned}
        G_W\lp m, \lambda \rpz & = \begin{pmatrix}
        1 & 0     \\
        0 & \frac{\mathbb{E}_{p_{m,\lambda}}x^2-2m\mathbb{E}_{p_{m,\lambda}}x+m^2}{\lambda^2}
    \end{pmatrix},  \\
    G_F\lp m, \lambda \rpz & = \frac{1}{\lambda^2}\begin{pmatrix}
        \int_{\mathbb{R}} \frac{p'^2}{p} dx & \int_{\mathbb{R}}  \frac{(x - m)p'^2}{\lambda p} dx     \\
        \int_{\mathbb{R}}  \frac{(x - m)p'^2}{\lambda p} dx & 1 + \int_{\RR} \lp \frac{\lp x - m \rpz^2 p'^2}{\lambda^2 p} + \frac{\lp x - m \rpz p'}{\lambda}\rpz dx 
    \end{pmatrix},
\end{aligned}
\end{equation*}
which illustrates that WIMs are simpler than Fisher information matrices in location-scale families.
\end{remark}
\end{example}
\section{Functional inequalities via information matrices}\label{sec-Func}
\par

In this section, we explore connections between information matrices and functional inequalities such as log-Sobolev inequalities (LSIs) and Poincar\'e inequalities (PIs) in statistical models. In section \ref{Efficiency}, we show that these inequalities are important for the study of statistical efficiency properties.
\subsection{Classical functional inequalities}
\par
Before working in statistical models, we first give a summary of relations among PIs, LSIs and dynamical quantities on density manifold.

Consider the relative entropy (KL-divergence) defined on density manifold:
    \begin{equation*}
        H(\mu|\nu) = \int_{\mathcal{X}}  \log \frac{\mu(x)}{\nu(x)} \mu(x)dx, \qquad \mu \in \mathcal{P}(\mathcal{X}).
    \end{equation*}
Here, we use a notation $H\lp \cdot | \cdot \rpz$ in order to be consistent with literature. We recall the definition of log-Sobolev inequality as below.
\begin{definition}[Log-Sobolev inequality]
    A probability measure $\nu$ is said to satisfy a log-Sobolev inequality with constant $\alpha > 0$ (in short: LSI$(\alpha)$) if we have:
    \begin{equation*}
        H(\mu|\nu) < \frac{1}{2\alpha}I(\mu|\nu), \qquad \mu \in \mathcal{P}(\mathcal{X}),
    \end{equation*}
    where the quantity $I(\mu|\nu)$ is the so-called Fisher-information functional 
    \begin{equation*}
        I(\mu|\nu) = \int_{\mathcal{X}} \lv \nabla_x \log \frac{\mu(x)}{\nu(x)} \rv^2 \mu(x) dx, \qquad \mu \in \mathcal{P}(\mathcal{X}).
    \end{equation*}
\end{definition}
\begin{remark}
    If we assume that $\mu$ is absolutely continuous w.r.t. the reference measure $\nu$ and define function $h$ on $\mathcal{X}$ as:
    \begin{equation*}
        \mu(x) = \frac{h(x)\nu(x)}{\int_{\mathcal{X}} h(x) \nu(x)dx},
    \end{equation*}
    then above definition of LSI translates to:
    \begin{equation*}
        \begin{aligned}
           & \lp H(\mu|\nu) \int_{\mathcal{X}}h(x)\nu(x) dx \rpz \\= &   \int_{\mathcal{X}}h(x) \log h(x) \nu(x) dx  - \lp \int_{\mathcal{X}}h(x) \nu(x) dx \rpz \log \lp \int_{\mathcal{X}}h(x) \nu(x) dx \rpz   \\
            \leq & \frac{1}{2\alpha}\int_{\mathcal{X}} \frac{\lv \nabla_x h(x) \rv^2}{h(x)} \nu(x) dx = \frac{1}{2\alpha}\lp I(\mu|\nu) \int_{\mathcal{X}}h(x) \nu(x) dx \rpz.
        \end{aligned}
    \end{equation*}
    The middle inequality is a more familiar definition of LSI$\lp \alpha \rpz$. By linearizing above formula with $h = 1 + \epsilon f, \epsilon \rightarrow 0$, we get the classical definition of PI$\lp \alpha \rpz$
    \begin{equation*}
        \int_{\mathcal{X}}f^2(x) \nu(x) dx   
        \leq \frac{1}{\alpha}\int_{\mathcal{X}} \lv \nabla_x f(x) \rv^2 \nu(x) dx, \qquad \int_{\mathcal{X}}f(x) \nu(x) dx  = 0.
    \end{equation*}
\end{remark}
\begin{definition}[Poincar\'e inequalities]
    A probability measure $\nu$ is said to satisfy a Poincar\'e inequalities with constant $\alpha > 0$ (in short: PI$(\alpha)$) if we have:
    \begin{equation*}
        \int_{\mathcal{X}}f^2(x) \nu(x) dx   
        \leq \frac{1}{\alpha}\int_{\mathcal{X}} \lv \nabla_x f(x) \rv^2 \nu(x) dx, \qquad \forall f, \  s.t. \int_{\mathcal{X}}f(x) \nu(x) dx  = 0.
    \end{equation*}
\end{definition}
\par
A sufficient criterion that guarantees LSIs and PIs is related to information matrices (operators or metrics in infinite dimension case) $G_W$.
\begin{proposition}\label{Hess-cond}
Denote $\Hessian_W H(\mu|\nu), G_W(\mu)$ two bi-linear forms correspondent to Hessian of the relative entropy and Wasserstein metric. 
\par
(1) Suppose $\Hessian_W H(\mu|\nu) - 2\alpha G_W(\mu)$ is a semi-positive definite bi-linear form on the Hilbert space $T_{\mu}\cp\lp \cx \rpz$, $\forall \mu \in \cp\lp \cx \rpz$. Then LSI$\lp \alpha \rpz$ holds for $\nu$.
\par
(2) Suppose $\Hessian_W H(\nu|\nu) - 2\alpha G_W(\nu)$ is a semi-positive definite bi-linear form on the Hilbert space $T_{\nu}\cp\lp \cx \rpz$. Then PI$\lp \alpha \rpz$ holds for $\nu$.
\end{proposition}
\begin{proof}
First, we prove the result concerned with LSIs. We compute the gradient of the relative entropy w.r.t. Wasserstein metric, which is given by:
\begin{equation*}  
    \begin{aligned}
        \grad_W H(\mu|\nu) = & - \nabla \cdot \left( \mu \nabla \frac{\delta}{\delta \mu}H(\mu|\nu) \right) = - \nabla \cdot \left( \mu \nabla \log \frac{\mu(x)}{\nu(x)} \right),
    \end{aligned}
\end{equation*}
where $\frac{\delta}{\delta \mu}$ refers to the $L^2$ functional derivative. Thus it is easy to obtain the relative entropy dissipation along the gradient flow as:
\begin{equation}\label{entro-diss}
    \begin{aligned}
    & \ \frac{d}{dt} H(\mu|\nu) \\
    = & \ - g_W \lp \grad_W H(\mu|\nu), \grad_W H(\mu|\nu) \rpz \\
    = & \ - \int_{\mathcal{X}} \lv \nabla_x \log \frac{\mu(x)}{\nu(x)} \rv^2 \mu(x) dx = - I(\mu|\nu).
    \end{aligned}
\end{equation}
Using the assumption, we have:
\begin{equation*}
    \begin{aligned}
        \frac{d^2}{dt^2} H(\mu_t|\nu) = & \  \Hessian_W H(\mu_t|\nu)\lp \grad_W H(\mu_t|\nu) , \grad_W H(\mu_t|\nu) \rpz           \\
        \geq & \ 2\alpha G_W(\mu_t) \lp \grad_W H(\mu_t|\nu) , \grad_W H(\mu_t|\nu) \rpz \\
        = & \  - 2\alpha \frac{d}{dt} H(\mu_t|\nu),
    \end{aligned}
\end{equation*}
from which LSI$(\alpha)$ holds via integrating the above formula, i.e.
\begin{equation*}
    \begin{aligned}
        & \ I(\mu_t|\nu) = I(\mu_t|\nu) - I(\nu|\nu)           \\
        = & \ \int_t^{\infty} \lp  \frac{d^2}{dt^2} H(\mu_{\tau}|\nu) \rpz d\tau      \\
        \geq & \ 2\alpha\int_t^{\infty} \lp  -\frac{d}{dt} H(\mu_{\tau}|\nu) \rpz d\tau       \\
        = & \ 2\alpha\lp H(\mu_t|\nu) - H(\nu|\nu) \rpz          \\
        = & \ 2\alpha H(\mu_t|\nu),
    \end{aligned}
\end{equation*}
where we use the fact that this gradient flow $\mu_t$ converges to $\nu$ and $H(\nu|\nu) = I(\nu|\nu) = 0$.
\par
To prove the conclusion of Poincar\'e inequalities, we consider a path in density manifold, i.e $\rho\lp \epsilon \rpz = \nu\lp 1 + \epsilon f \rpz, \int_{\mathcal{X}}f(x) \nu(x) dx  = 0$. Since we have
\begin{equation*}
    \begin{aligned}
        H\lp \rho\lp \epsilon \rpz| \nu \rpz & = \frac{\epsilon^2}{2}\int_{\mathcal{X}}f^2(x) \nu(x) dx + o\lp \epsilon^2 \rpz,           \\  
    - \frac{d}{dt} H(\rho\lp \epsilon \rpz|\nu) = I\lp \rho\lp \epsilon \rpz| \nu \rpz & = \epsilon^2\int_{\mathcal{X}} \lv \nabla_x f(x) \rv^2 \nu(x) dx + o\lp \epsilon^2 \rpz.
    \end{aligned}
\end{equation*}
Consequently, we obtain
\begin{equation*}
    \begin{aligned}
        & \ \frac{\int_{\mathcal{X}}f^2(x) \nu(x) dx}{\int_{\mathcal{X}} \lv \nabla_x f(x) \rv^2 \nu(x) dx}  \\
        = & \ \frac{1}{2}\lim_{\epsilon \rightarrow 0} - \frac{H\lp \rho\lp \epsilon \rpz| \nu \rpz}{\frac{d}{d\epsilon} H(\rho\lp \epsilon \rpz|\nu)}      \\
        = & \ \frac{1}{2}\lim_{\epsilon \rightarrow 0} - \frac{\frac{d}{d\epsilon}H\lp \rho\lp \epsilon \rpz| \nu \rpz}{\frac{d^2}{d\epsilon^2} H(\rho\lp \epsilon \rpz|\nu)}      \\
        = & \ \frac{1}{2}\lim_{\epsilon \rightarrow 0} \frac{G_W\lp \rho\lp 0 \rpz \rpz \lp \frac{d}{d\epsilon}\rho\lp 0 \rpz, \frac{d}{d\epsilon}\rho\lp 0 \rpz \rpz}{\Hessian_W H(\rho\lp 0 \rpz|\nu)\lp \frac{d}{d\epsilon}\rho\lp 0 \rpz, \frac{d}{d\epsilon}\rho\lp 0 \rpz \rpz}      \\ 
        \leq & \ \frac{1}{\alpha},
    \end{aligned}
\end{equation*}
where we use L'Hopital's rule in second equality and the third equality holds because of the assumption that $\Hessian_W H(\nu|\nu) - 2\alpha G_W(\nu)$ is semi-definite.
\end{proof}
\begin{remark}
    With the help of \eqref{entro-diss}, readers can recognize that LSI guarantees a global exponential convergence of the gradient flow of the relative entropy $H\lp \cdot | \nu \rpz$. Indeed, suppose $\mu_t$ is a gradient flow of $H(\cdot|\nu)$ starting from $\mu_0$, then we have:
    \begin{equation*}
        \begin{aligned}
            H(\mu_t|\nu) & \leq e^{-2\alpha t}H(\mu_0|\nu),\qquad \mu_0 \in \mathcal{P}(\mathcal{X})    \text{ (LSI$\lp \alpha \rpz$)} .
        \end{aligned}
    \end{equation*}
    While intuitively speaking, a PI can be viewed as an infinitesimal version of a LSI, that is to consider the dynamics in a neighborhood of the optimal value.
\end{remark}

\subsection{LSIs and PIs in families}
\par
Now, it is clear that PIs and LSIs are related to density manifold. Here, we attempt to find those counterparts in statistical models, i.e. submanifolds.
\par
Now, we fix a model $\Theta \subset \mathcal{P}(\mathcal{X})$ with metric given by $G_W$. The relative entropy $H\lp \cdot|\nu \rpz$ is indeed a restriction of global functional to this family. And we furthermore require the reference measure $\nu$ to lie in this family, i.e. $\nu = p_{\theta_*},\theta_* \in \Theta$. We use  $\wtd {}$ to distinguish constraint cases (statistical models) from the global situation (density manifold). Recall that the Fisher information functional is merely the relative entropy dissipation along a gradient flow. Thus we have
\begin{equation}
    \begin{aligned}
        \wtd I(p_{\theta_t}|p_{\theta_*}) =& - \frac{d}{dt}\wtd H(p_{\theta_t}|p_{\theta_*})\\
        = & \  g_W\lp \grad_W \wtd H(p_\theta|p_{\theta_*}), \grad_W \wtd H(p_\theta|p_{\theta_*}) \rpz        \\
        = & \ \nabla_{\theta} \wtd H^T \lp \wtd G_W^{-1} \rpz^T \wtd G_W \wtd  G_W^{-1}\nabla_{\theta} \wtd H \\
        = & \  \nabla_{\theta} \wtd H^T \wtd G_W^{-1}\nabla_{\theta} \wtd H,
    \end{aligned}
\end{equation}
where we use a fact
\begin{equation*}
    \grad_W \wtd H(p_\theta|p_{\theta_*}) = \wtd G_W^{-1}\nabla_{\theta} \wtd H.
\end{equation*}
\begin{definition}[LSI in family]
    Consider a statistical model $p: \mathcal{X} \times \Theta \rightarrow \mathbb{R}$, a probability measure $p_{\theta_*}$ is said to satisfy LSI$(\alpha)$ in $\Theta$ with constant $\alpha > 0$ (in short: LSI$(\alpha)$) if we have:
    \begin{equation*}
        \wtd H(p_\theta|p_{\theta_*}) < \frac{1}{2\alpha}\wtd I(p_\theta|p_{\theta_*}), \qquad \theta \in \Theta.
    \end{equation*}
\end{definition}
\par
Using information matrices, we seek a sufficient condition for LSIs and PIs as proposition \ref{Hess-cond}:
\begin{equation*}
    \Hessian_W \wtd H(p_\theta|p_{\theta_*}) \geq 2\alpha \wtd G_W(\theta),
\end{equation*}
where we have to take care that the Hessian on LHS is calculated in a submanifold instead of density manifold. Fisher information matrix also comes into this picture, via a decomposition of the Hessian term $\Hessian_W \wtd H(p_\theta|p_{\theta_*})$. This point is known as the Ricci-information-Wasserstein (RIW) condition.
\begin{theorem}[RIW-condition]\label{RIW}
    The information matrices criterion for LSI$\lp \alpha \rpz$ of distribution $p_{\theta*}$ is given by:
    \begin{equation*}
        G_F\lp \theta \rpz +  \nabla_{\theta}^2p_\theta\log \frac{p_\theta}{p_{\theta_*}} - \Gamma^W \nabla_{\theta}\wtd H(p_\theta|p_{\theta_*}) \geq 2\alpha G_W(\theta),
    \end{equation*}
    where $\Gamma^W$s are Christoffel symbols in Wasserstein statistical model $\Theta$, while for PI$\lp \alpha \rpz$ of distribution $p_{\theta*}$ can be written as:
    \begin{equation*}
        G_F\lp \theta \rpz +  \nabla_{\theta}^2p_\theta\log \frac{p_\theta}{p_{\theta_*}} \geq 2\alpha G_W(\theta).
    \end{equation*}
\end{theorem}
\begin{remark}
    It can be seen that the condition for log-Sobolev inequalities is much more complicated than that of Poincar{\'e} inequalities. For LSIs require a global convexity of the entropy while PIs only correspond to local behavior at the minimum. The most significant change takes place in the Hessian term of entropy, where Wasserstein Christoffel symbols come in.
\end{remark}

\subsection{Examples in 1-d Family}]\label{ex-func}
Both LSIs and PIs can be proved by using Wasserstein and Fisher information matrices. Previously, we have done geometric computations on metric tensor and Hessian of the entropy. This prepares ingredients for us to establish inequalities in families of probability distributions. In this section, we utilize previous calculations to obtain concrete bounds on these functional inequalities.

\begin{example}[Gaussian distribution]      
Recall that for a Gaussian distribution with mean value $\mu$ and standard variance $\sigma$, the Wasserstein and Fisher information matrices are given by:
\begin{equation*}
    G_W(\mu,\sigma)  = \begin{pmatrix}
	1 & 0		\\
	0 & 1
	\end{pmatrix}, \qquad 
    G_F(\mu,\sigma) = \begin{pmatrix}
			\frac{1}{\sigma^2} & 0		\\
			0 & \frac{2}{\sigma^2}
	\end{pmatrix}.
\end{equation*}
The entropy and the relative entropy defined on this model are provided by:
\begin{equation*}
    \begin{aligned}
        \wtd H(\mu,\sigma) & = - \half \log 2\pi - \log \sigma - \half ,           \\
        \wtd H(\mu,\sigma|p_*) & = - \log \sigma + \log \sigma_* - \half + \frac{\sigma^2 + \lp \mu - \mu_* \rpz^2}{2\sigma_*^2},\qquad p_* \sim p_{\mu_*,\sigma_*}.
    \end{aligned}
\end{equation*}
We can calculate Wasserstein gradients associated with these two functionals:
\begin{equation*}
    \begin{aligned}
        \nabla_{\mu,\sigma}^W \wtd H(\mu,\sigma) = \left(
            \begin{aligned}
                & 0   \     \\
                - & \frac{1}{\sigma} \ 
            \end{aligned}\right),   
\quad        \nabla_{\mu,\sigma}^W \wtd H(\mu,\sigma|p_*) = \left(
            \begin{aligned}
                & \frac{\mu - \mu_*}{\sigma_*^2}     \   \\
                - & \frac{1}{\sigma} + \frac{\sigma}{\sigma_*^2} \ 
            \end{aligned}\right),
    \end{aligned}
\end{equation*}
with the correspondent Fisher information functionals as:
\begin{equation*}
    \begin{aligned}
        \wtd I(\mu,\sigma) & = \frac{1}{\sigma^2},          \\
        \wtd I(\mu,\sigma|p_*) & = \frac{\lp \mu - \mu_* \rpz^2}{\sigma_*^4} + \lp - \frac{1}{\sigma} + \frac{\sigma}{\sigma_*^2} \rpz^2.
    \end{aligned}
\end{equation*}
Thus, the LSI$\lp \alpha \rpz$ for Gaussian $p_{\mu_*,\sigma_*}$ is given by
\begin{equation*}
    \begin{aligned}
        \frac{\lp \mu - \mu_* \rpz^2}{\sigma_*^4} + \lp - \frac{1}{\sigma} + \frac{\sigma}{\sigma_*^2} \rpz^2 & \geq 2\alpha \lp - \log \sigma + \log \sigma_* - \half + \frac{\sigma^2 + \lp \mu - \mu_* \rpz^2}{2\sigma_*^2} \rpz.
    \end{aligned}
\end{equation*}
\par
Next, we move onto the derivation of the RIW condition. It suffices to consider a relation between $\wtd G_W, \Hessian_W \wtd H$ at each point in a statistical model. Recall the formula for Hessian in Riemannian geometry:
\begin{equation*}
    \lp \Hessian f \rpz_{ij} = \pa_i\pa_j f - \Gamma_{ij}^{k(W)} \pa_k f,
\end{equation*}
where $\Gamma^{W}$s are Christoffel symbols in Wasserstein geometry. In Wasserstein Gaussian model where the metric is Euclidean, Christoffel symbols vanish, i.e. $\Gamma^{W} = 0$. Thus we have:
\begin{equation*}
    \Hessian_W \wtd H\lp \mu, \sigma \rpz = \begin{pmatrix}
	0 & 0		\\
	0 & \frac{1}{\sigma^2}
	\end{pmatrix}, \qquad \Hessian_W \wtd H\lp \mu, \sigma | p_{*} \rpz = \begin{pmatrix}
	\frac{1}{\sigma_*^2} & 0		\\
	0 & \frac{1}{\sigma^2} + \frac{1}{\sigma_*^2}
	\end{pmatrix}.
\end{equation*}

\par
For a gradient flow of the relative entropy w.r.t. a Gaussian $p_{\mu_*, \sigma_*}$, we conclude that
\begin{equation*}
    \Hessian_W \wtd H(\mu,\sigma|p_{\theta_*}) \geq \lp \frac{1}{\sigma_*^2} \rpz G_W(\mu,\sigma),
\end{equation*}
since $G_W(\mu,\sigma)$ is exactly an identity matrix. In other words, the Gaussian $p_{\mu_*,\sigma_*}$ satisfies a LSI$\lp \frac{1}{2\sigma_*^2}\rpz$ in a Gaussian model. Notice this result coincides with the one in global case.
\par
Next, for the gradient flow of the entropy $\wtd H\lp \cdot \rpz$, we do not have a satisfying constant $\alpha$ such that the Hessian condition proposition \ref{Hess-cond} holds. For $\Hessian_W \wtd H(\mu,\sigma)$ matrix has an eigenvalue $0$. 
Despite of this, we have:
\begin{equation*}
    \grad_W \wtd H(\mu,\sigma) = G_W^{-1}\nabla_{\mu,\sigma} \wtd H(\mu,\sigma) = \nabla_{\mu,\sigma}\wtd H(\mu,\sigma),
\end{equation*}
whose $\mu$ component always vanishes. Thus the gradient direction of $\wtd H(\cdot)$ always coincides with $\sigma$ direction, in which we have eigenvalue's bound: $\eig_{\sigma}(\wtd H) \geq \frac{1}{\sigma^2} \eig_{\sigma}(G_W)$. This refers to that eigenvalues of two matrices correspond to direction $\frac{\pa}{\pa\mu}$ have a bound. For LSIs, if the range of $\sigma$ is the whole $\RR$, then it is easy to see there will not exist a satisfying constant $\alpha > 0$ for LSI$(\alpha)$ to hold, i.e. $\frac{1}{\sigma^2} \geq 2\alpha, \ \forall \sigma \in \mbR$. However, if we restrict the range of $\sigma$ to a bounded region such as $[-M,M]$, then LSI$(\frac{1}{2M^2})$ will hold.
\begin{remark}
    Above calculation on gradient flows of the entropy does not establish LSI$(\alpha)$ for any specific distribution. It merely provides an example of using Hessian condition to study dynamical behaviors.
\end{remark}
\end{example}
\begin{example}[Laplacian distribution]
Consider the case of Laplacian distribution, where
\begin{equation*}
    G_W(m, \lambda)  = \begin{pmatrix}
	1 & 0		\\
	0 & \frac{2}{\lambda^4}
	\end{pmatrix}, \quad
	G_F(m, \lambda)  = \begin{pmatrix}
	\lambda^2 & 0		\\
	0 & \frac{1}{\lambda^2}
	\end{pmatrix},
\end{equation*}
from which we can calculate the Christoffel symbol as:
\begin{equation*} 
    \begin{aligned}
    \Gamma_{22}^{2(W)}(m, \lambda) & = \frac{g^{-1}_{22}}{2} (\pa_2 g_{22} + \pa_2 g_{22} - \pa_2 g_{22}) = \frac{g^{-1}_{22}}{2} \pa_2 g_{22} = \frac{\lambda^4}{4} \cdot \lp - \frac{8}{\lambda^5} \rpz = - \frac{2}{\lambda},      \\
    \Gamma_{ij}^{k(W)}(m, \lambda) & = 0 \qquad \text{otherwise}.
    \end{aligned}
\end{equation*}
Following the same procedure we have done before, the entropy and the relative entropy w.r.t. $p_{m_*, \lambda_*}$ defined on this model is provided by:
\begin{equation*}
    \begin{aligned}
        \wtd H(m,\lambda) & = - 1 + \log \lambda - \log 2 ,           \\
        \wtd H(m,\lambda|p*) & = - 1 + \log \lambda - \log \lambda_* + \lambda_*\lv m - m_* \rv + \frac{\lambda_* e^{ - \lambda\lv m - m_* \rv}}{\lambda},    \end{aligned}
\end{equation*}
from which we can calculate Wasserstein gradients associated with two functionals:
\begin{equation*}
    \begin{aligned}
        & \nabla_{m,\lambda}^W \wtd H(m,\lambda) = \left(
            \begin{aligned}
                \ & 0   \     \\
                \ & \frac{1}{\lambda} \ 
            \end{aligned}\right),   \\  
        & \nabla_{m,\lambda}^W \wtd H(m,\lambda|p_*) =
        \lbb
        \begin{aligned}
            & \left(
            \begin{aligned}
                & \lambda_*\lp 1 - e^{ - \lambda\lp m - m_* \rpz}\rpz         \\
                &  - \frac{\lp \lambda\lp m - m_* \rpz + 1 \rpz\lambda_*e^{ - \lambda\lp m - m_* \rpz} - \lambda}{\lambda^2} 
            \end{aligned}\right),  \ m > m_*,      \\
            & \left(
            \begin{aligned}
                & - \lambda_*\lp 1 - e^{ - \lambda\lp m_* - m \rpz}\rpz         \\
                & - \frac{\lp \lambda\lp m_* - m \rpz + 1 \rpz\lambda_*e^{ - \lambda\lp m_* - m \rpz} - \lambda}{\lambda^2}  
            \end{aligned}\right),  \ m < m_*,
        \end{aligned}\right.
    \end{aligned}
\end{equation*}
with the Fisher information functionals as:
\begin{equation*}
    \begin{aligned}
        \wtd I(m,\lambda) & = \frac{\lambda^2}{2},          \\
        \wtd I(m,\lambda|p_*) & = 
            \lambda_*^2\lp 1 - e^{ - \lambda\lv m - m_* \rv}\rpz^2 + \frac{\lp \lp \lambda\lv m - m_* \rv + 1 \rpz\lambda_*e^{ - \lambda\lv m - m_* \rv} - \lambda \rpz^2}{2}.
    \end{aligned}
\end{equation*}
Notice that the value of $\nabla_{m,\lambda}^W \wtd  H(m,\lambda|p_*)$ is not well-defined at point $m = m_*$. However, what we considered is integral on the whole $\RR$. Thus we can simply ignore its value at $m = m_*$. As before, LSI$\lp \alpha \rpz$ is given by
\begin{equation*}
    \begin{aligned}
            \ & \lambda_*^2\lp 1 - e^{ - \lambda\lv m_* - m \rv}\rpz^2 + \frac{\lp \lp \lambda\lv m_* - m \rv + 1 \rpz\lambda_*e^{ - \lambda\lv m_* - m \rv} - \lambda \rpz^2}{2}      \\
            \geq \  & 2\alpha \lp - 1 + \log \lambda - \log \lambda_* + \lambda_*\lv m - m_* \rv + \frac{\lambda_* e^{ - \lambda\lv m - m_* \rv}}{\lambda} \rpz.
    \end{aligned}
\end{equation*}
And we find Hessians of the entropy and the relative entropy in $(\Theta, G_W)$ are given by:
\begin{equation*}
    \begin{aligned}
        \Hessian_W \wtd H\lp m,\lambda \rpz & = \begin{pmatrix}
	0 & 0		\\
	0 & \frac{1}{\lambda^2}
	\end{pmatrix},      \\ 
	\Hessian_W \wtd H\lp m,\lambda | p_* \rpz & = \begin{pmatrix}
	\lambda\lambda_* e^{ - \lambda\lv m - m_* \rv} & 0		\\
	0 & \frac{1}{\lambda^2} + \frac{\lambda_* e^{ - \lambda\lv m - m_* \rv}\lp m_* - m \rpz^2 }{\lambda^3}
	\end{pmatrix}.
    \end{aligned}
\end{equation*}
Following the same analysis, we conclude that for gradient flows of the entropy $\wtd H(m,\lambda)$, a LSI$(\frac{\lambda^2}{4})$ holds. While for the relative entropy $H(m,\lambda|p_{m_*, \lambda_*})$, Hessian condition can be written as
\begin{equation*}
    \begin{pmatrix}
	\lambda\lambda_* e^{ - \lambda\lv m - m_* \rv} & 0		\\
	0 & \frac{1}{\lambda^2} + \frac{\lambda_* e^{ - \lambda\lv m - m_* \rv}\lp m_* - m \rpz^2 }{\lambda^3}
	\end{pmatrix} \geq \alpha \begin{pmatrix}
	1 & 0		\\
	0 & \frac{2}{\lambda^4}
	\end{pmatrix},
\end{equation*}
which can be reformulated as
\begin{equation*}
    \begin{aligned}
        \alpha = \min_{m,\lambda}\lbb \lambda\lambda_* e^{ - \lambda\lv m - m_* \rv}, \half\lp \lambda^2 + \lambda_* e^{ - \lambda\lv m - m_* \rv}\lambda\lp m_* - m \rpz^2  \rpz \rbb.
    \end{aligned}
\end{equation*}
From above formula, we conclude that in order to find a satisfying constant, it suffices to restrict the region of $m \in \lb -M, M \rb,\lambda \in \lb N, \infty \rpz$. The distribution La($m_*, \lambda_*$) satisfies a LSI$(\alpha)$ in Laplacian family with $\alpha$ given above.
\end{example} 
\begin{example}[Independent model]
For an independent family $p\lp x, y; \theta \rpz = p_1\lp x; \theta \rpz p\lp y; \theta \rpz$, we have
\begin{equation*}
    G_W = G_W^1 + G_W^2, \qquad G_F = G_F^1 + G_F^2.
\end{equation*}
The entropy and the relative entropy also have this separability property 
\begin{equation*}
    \begin{aligned}
        \wtd H\lp \theta \rpz & = \wtd H_1\lp \theta \rpz + \wtd H_2\lp \theta \rpz,   \\
        \wtd H\lp \theta | p_*\rpz & = \wtd H_1\lp \theta | p_{1*}\rpz + \wtd H_2\lp \theta |p_{2*} \rpz,       \\
        \nabla_{\theta} \wtd H\lp \theta | p_*\rpz & = \nabla_{\theta}\wtd H_1\lp \theta | p_{1*}\rpz + \nabla_{\theta}\wtd H_2\lp \theta |p_{2*} \rpz.
    \end{aligned}
\end{equation*}
The Fisher information functional is given by
\begin{equation*}
    \begin{aligned}
    & \ I\lp p_{\theta}| p_* \rpz \\
    = & \  \lp \nabla_{\theta}\wtd H_1\lp \theta | p_{1*}\rpz + \nabla_{\theta}\wtd H_2\lp \theta |p_{2*} \rpz \rpz^T \lp G_W^1 + G_W^2 \rpz^{-1} \lp \nabla_{\theta}\wtd H_1\lp \theta | p_{1*}\rpz + \nabla_{\theta}\wtd H_2\lp \theta |p_{2*} \rpz \rpz,
    \end{aligned}
\end{equation*}
with LSI$\lp \alpha \rpz$ given by
\begin{equation*}
    \begin{aligned}
        & \ \lp \nabla_{\theta}\wtd H_1\lp \theta | p_{1*}\rpz + \nabla_{\theta}\wtd H_2\lp \theta |p_{2*} \rpz \rpz^T \lp G_W^1 + G_W^2 \rpz^{-1} \lp \nabla_{\theta}\wtd H_1\lp \theta | p_{1*}\rpz + \nabla_{\theta}\wtd H_2\lp \theta |p_{2*} \rpz \rpz \\ \geq
        & \ 2\alpha \nabla_{\theta}\wtd H_1\lp \theta | p_{1*}\rpz + \nabla_{\theta}\wtd H_2\lp \theta |p_{2*} \rpz.
    \end{aligned}
\end{equation*}
\end{example}
\par
In conclusion, above examples introduce another way to prove functional inequalities as well as convergence rates of dynamics in probability families.
\section{Proofs in section \ref{Efficiency}}
\subsection{Proof of Theorem \ref{update}}\label{proof-main}
\begin{proof}[Proof of Theorem \ref{update}]
First, we postulate that $\nabla_x$ refers to the gradient w.r.t. $x$ variable while $\nabla_{\theta}$ refers to the gradient w.r.t. $\theta$ variable. We expand the function $l(x_t,\theta_t)$
\begin{equation*}
    l(x_t,\theta_t) = l(x_t,\theta_*) + \nabla_{\theta}l(x_t,\theta_*)\left(\theta_t - \theta_* \right) + O\left( \lv \theta_t - \theta_* \rv^2 \right).
\end{equation*}
By substrating $\theta_*$ in both sides of the updating equation and plugging in the expansion above, we get:
\begin{equation*}
    \begin{aligned}
    \theta_{t+1} - \theta_* = & \ \left( \theta_{t} - \theta_* \right) - \frac{1}{t}G_W^{-1}(\theta_t)\left(l(x_t,\theta_*) + \nabla_{\theta}l(x_t,\theta_*)\left(\theta_t - \theta_* \right)\right. \\
    & \ +  \left.O\left( \lv \theta_t - \theta_* \rv^2 \right)\right).
    \end{aligned}
\end{equation*}
Then, taking Wasserstein covariances of both sides, we get:
\begin{equation*}  
\begin{aligned}
V_{t+1} = V_t & + \frac{1}{t^2}\mbE_{p_{\theta_*}}  \lb  \nabla_x \lp G_W^{-1}(\theta_t) l(x_t,\theta_t) \rpz \cdot \nabla_x  \lp l(x_t,\theta_t)^{T} G_W^{-1}(\theta_t) \rpz   \rb \\
& - \frac{2}{t} \mbE_{p_{\theta_*}} \lb \nabla_x \left(\theta_t - \theta_* \right) \cdot \nabla_x \lp l(x_t,\theta_*)^{T} G_W^{-1}(\theta_t) \rpz \rb + o\lp \frac{V_t}{t} \rpz \\
&  - \frac{2}{t} \mbE_{p_{\theta_*}} \lb \nabla_x \left(\theta_t - \theta_* \right)  \cdot \nabla_x \left( \left(\theta_t - \theta_* \right)^T\nabla_{\theta}l(x_t,\theta_*)^T G_W^{-1} (\theta_t) \right) \rb,
\end{aligned}
\end{equation*}
where the last term corresponds to an expansion term $O\left( \lv \theta_t - \theta_* \rv^2 \right)$ and we use an assumption that $\mbE_{p_{\theta_*}} \lb \lp \theta_t - \theta_* \rpz^2 \rb,
\mbE_{p_{\theta_*}} \lb \lv \nabla_x \lp \theta_t - \theta_* \rpz \rv^2 \rb = o(1)$. In above formula, we eliminate transpose symbols ${}^T$ on metric tensor $G_W$ because of its symmetry. For the second term on the RHS, we have:
\begin{equation*}
	\begin{aligned}
		& \ \frac{1}{t^2}\mbE_{p_{\theta_*}} \lb  \nabla_x \lp G_W^{-1}(\theta_t) l(x_t,\theta_t) \rpz \cdot \nabla_x  \lp l(x_t,\theta_t)^{T} G_W^{-1}(\theta_t) \rpz   \rb	\\
		= & \ \frac{1}{t^2}\mbE_{p_{\theta_*}} \lb G_W^{-1}(\theta_*)\nabla_x \lp l(x_t,\theta_*) \rpz \cdot \nabla_x \lp l(x_t,\theta_*)^{T} \rpz G_W^{-1}(\theta_*)  \rb + o \lp \frac{1}{t^2} \rpz			\\
		= & \ \frac{1}{t^2} G_W^{-1}(\theta_*) \mbE_{p_{\theta_*}} \lb \nabla_x \lp l(x_t,\theta_*) \rpz \cdot \nabla_x \lp l(x_t,\theta_*)^{T} \rpz  \rb G_W^{-1}(\theta_*) + o \lp \frac{1}{t^2} \rpz,  
	\end{aligned}
\end{equation*}
where we use the following fact
\begin{equation*}
    \begin{aligned}
    	& \ \mbE_{p_{\theta_*}} \lb  \nabla_x \lp G_W^{-1}(\theta_t) l(x_t,\theta_t) \rpz \cdot \nabla_x  \lp l(x_t,\theta_t)^{T} G_W^{-1}(\theta_t) \rpz   \rb \\
    	& \ -  \mbE_{p_{\theta_*}} \lb     G_W^{-1}(\theta_*)\nabla_x \lp l(x_t,\theta_*) \rpz \cdot \nabla_x \lp l(x_t,\theta_*)^{T} \rpz G_W^{-1}(\theta_*)  \rb \\
    	= & \ O\lp \EE_{p_{\theta_*}} \lv \theta_t - \theta_* \rv \rpz = o\lp 1 \rpz.
    \end{aligned}
\end{equation*}
And the third term in the RHS can be reduced according to:
\begin{equation*}
	\begin{aligned}
		& - \frac{2}{t} \mbE_{p_{\theta_*}}\lb \nabla_x \left(\theta_t - \theta_* \right) \cdot \nabla_x \lp l(x_t,\theta_*)^{T} G_W^{-1}(\theta_t) \rpz \rb		\\
		= & - \frac{2}{t} \mbE_{p_{\theta_*}} \lb \nabla_x \left(\theta_t - \theta_* \right) \cdot \nabla_x \lp l(x_t,\theta_*)^{T} \rpz G_W^{-1}(\theta_t) \rb \\
		& \ - \frac{2}{t} 	\mbE_{p_{\theta_*}} \lb \nabla_x \left(\theta_t - \theta_* \right) \cdot l(x_t,\theta_*)^{T} \nabla_x G_W^{-1}(\theta_t) \rb				\\
		= & \ 0,
	\end{aligned}
\end{equation*}
where the first term vanishes because $\nabla_x \left(\theta_t - \theta_* \right)$ only has non-vanishing components at $x_1,...,x_{t-1}$ while $\nabla_x \lp f(x_t,\theta_*)^{T} \rpz$ only has a non-vanishing component at $x_t$. 
Consequently their inner product vanishes everywhere. While the second term vanishes by considering each element of this matrix, we have:
\begin{equation*}
	\begin{aligned}
		& \lp \mbE_{p_{\theta_*}} \lb \nabla_x \left(\theta_t - \theta_* \right) \cdot l(x_t,\theta_*)^{T} \nabla_x G_W^{-1}(\theta_t) \rb \rpz_{ij} \\
		= & \  \frac{2}{t}\mbE_{p_{\theta_*}}\nabla_x \left(\theta_t - \theta_* \right)_i \cdot \lp l(x_t,\theta_*)^{T} \nabla_x G_W^{-1}(\theta_t) \rpz_j 		\\
		= & \  \frac{2}{t}\mbE_{p_{\theta_*}} \lb \nabla_x \left(\theta_t - \theta_* \right)_i  \cdot \nabla_x \lp G_W^{-1}(\theta_t) _{kj} \rpz l(x_t,\theta_*)^{T}_k  \rb			\\
		= &  \ \frac{2}{t}\mbE_{p_{\theta_*}}\lb \nabla_x \left(\theta_t - \theta_* \right)_i  \cdot \nabla_x \lp G_W^{-1}(\theta_t) _{kj} \rpz \rb \mbE_{p_{\theta_*}}l(x_t,\theta_*)^{T}_k  	\\
		= & \ 0,
	\end{aligned}
\end{equation*}
where the third equality is guaranteed by the fact that $\theta_t - \theta_*$ is independent to $\nabla_{\theta}f(x_t,\theta_*)$ since $\theta_t, x_t$ are mutually independent. While the last equality holds by an assumption: 
\begin{equation*}
	\mbE_{p_{\theta_*}}l(x_t,\theta_*) = 0.
\end{equation*}
For the last term, same as the analysis of the third term, we find:
\begin{equation*}
	\begin{aligned}
		& - \frac{2}{t} \mbE_{p_{\theta_*}} \lb \nabla_x \left(\theta_t - \theta_* \right)  \cdot \nabla_x \left( \left(\theta_t - \theta_* \right)^T\nabla_{\theta}l(x_t,\theta_*)^T G_W^{-1} (\theta_t) \right) \rb			\\
		= &  - \frac{2}{t} \mbE_{p_{\theta_*}} \lb \nabla_x \left(\theta_t - \theta_* \right)  \cdot \nabla_x \left( \left(\theta_t - \theta_* \right)^T \right) \nabla_{\theta}l(x_t,\theta_*)^T			G_W^{-1} (\theta_t) 	\rb	\\
		& - \frac{2}{t} \mbE_{p_{\theta_*}} \lb \nabla_x \left(\theta_t - \theta_* \right)  \cdot \left(\theta_t - \theta_* \right)^T  \nabla_x 					\left( \nabla_{\theta}l(x_t,\theta_*)^T \right)  G_W^{-1} (\theta_t) 	\rb		\\
		& - \frac{2}{t} \mbE_{p_{\theta_*}} \lb \nabla_x \left(\theta_t - \theta_* \right)  \cdot \left(\theta_t - \theta_* \right)^T \nabla_{\theta}l(x_t,\theta_*)^T	 \nabla_x \left(   G_W^{-1} (\theta_t) \right)	\rb	\\
		= & - \frac{2}{t} \mbE_{p_{\theta_*}} \lb \nabla_x \left(\theta_t - \theta_* \right)  \cdot \nabla_x \left( \left(\theta_t - \theta_* \right)^T \right) \nabla_{\theta}l(x_t,\theta_*)^T			G_W^{-1} (\theta_*) \rb + o(\frac{V_t}{t})		\\
		& - \frac{2}{t} \mbE_{p_{\theta_*}} \lb \nabla_x \left(\theta_t - \theta_* \right)  \cdot \left(\theta_t - \theta_* \right)^T \nabla_{\theta}l(x_t,\theta_*)^T	 \nabla_x \left(  G_W^{-1} (\theta_t)  \right) \rb,
	\end{aligned}
\end{equation*}
where we again use the independent relation between $\left(\theta_t - \theta_* \right)$ and $f(x_t,\theta_*)$. The additional term appearing above, with the help that $\mbE_{p_{\theta_*}} \lb \nabla_{\theta}f(x_t,\theta_*) \rb = O(1)$, $\nabla_x \left(  G_W^{-1} (\theta_t) \right) = \nabla_x \theta_t \nabla_{\theta} \lp G_W^{-1} (\theta_t) \rpz = O(\nabla_x \lp \theta_t - \theta_* \rpz)$, can be further reduced to the form below:
\begin{equation*}
    \begin{aligned}
        & \ \frac{2}{t} \mbE_{p_{\theta_*}} \lb \nabla_x \left(\theta_t - \theta_* \right)  \cdot \left(\theta_t - \theta_* \right)^T \nabla_{\theta}l(x_t,\theta_*)	 \nabla_x \left(  \left( G_W^{-1} (\theta_t) \right) \right) \rb         \\
        = & \ \frac{2}{t} \mbE_{p_{\theta_*}} \lb \nabla_x \left(\theta_t - \theta_* \right)  \cdot \left(\theta_t - \theta_* \right)^T O(1)	 \nabla_x \lp \theta_t - \theta_* \rpz O(1) \rb      \\
        \leq & \ \frac{O(1)}{t} \sqrt{\mbE_{p_{\theta_*}}  \lb \lv \nabla_x \lp \theta_t - \theta_* \rpz \rv^2 \rb \mbE_{p_{\theta_*}} \lb \lp \theta_t - \theta_* \rpz^2 \rb \mbE_{p_{\theta_*}} \lb \lv \nabla_x \lp \theta_t - \theta_* \rpz \rv^2 \rb}     \\
        = & \ o\lp \frac{V_t}{t} \rpz.
    \end{aligned}
\end{equation*}
And the last term finally reduces to:
\begin{equation*}
    \begin{aligned}
    	& - \frac{2}{t} \mbE_{p_{\theta_*}} \lb \nabla_x \left(\theta_t - \theta_* \right)  \cdot \nabla_x \left( \left(\theta_t - \theta_* \right)^T \right) \nabla_{\theta}l(x_t,\theta_*)			\left( G_W^{-1} (\theta_*) \right) \rb + o(\frac{V_t}{t})       \\
    	= & - \frac{2V_t}{t} \mbE_{p_{\theta_*}} \lb \nabla_{\theta}l(x_t,\theta_*) \rb G_W^{-1} (\theta_*) + o(\frac{V_t}{t}).
    \end{aligned}
\end{equation*}	
\indent
Combining all the terms we have in hand, we derive the following updating equation for Wasserstein covariances during a natural gradient descent:
\begin{equation*}
    \begin{aligned}
     V_{t+1} = & \ V_t + \frac{1}{t^2} G_W^{-1}(\theta_*) \mbE_{p_{\theta_*}} \lb \nabla_x \lp l(x_t,\theta_*) \rpz \cdot \nabla_x \lp l(x_t,\theta_*)^{T} \rpz  \rb \left( G_W^{-1}(\theta_*) \right) \\ 
     & - \frac{2V_t}{t} \mbE_{p_{\theta_*}} \lb \nabla_{\theta}l(x_t,\theta_*) \rb G_W^{-1} (\theta_*) + o(\frac{V_t}{t}) + o\left( \frac{1}{t^2} \right) + O(\frac{V_t}{t^2}).
	\end{aligned}
\end{equation*}
\end{proof}
\begin{remark}
    The most frequently used tools in this proof is a separability property, c.f. proposition \ref{sep}. The key observation here is that, for two statistics $T_1,T_2$ which depend on (independent) different variables, such as $T_1 = T_1(x_1,...,x_{t-1})$, $T_2 = T_2(x_t,...,x_{t+n})$ are ``orthogonal'' in both Wasserstein and Fisher metrics. Specifically, consider gradients of $T_1,T_2$ w.r.t. $x$, since they depend on different variables, thus
    \begin{equation*}
        \Cov^W\lb T_1, T_2 \rb = \EE_{p_{\theta_*}}\lb \nabla_x T_1 \cdot \nabla_x T_2 \rb = 0.
    \end{equation*}
    This type of separability is a direct analog of the one in Fisher-Rao geometry:
    \begin{equation*}
        \Cov^F\lb T_1, T_2 \rb = \EE_{p_{\theta_*}}\lb T_1 T_2 \rb= \EE_{p_{\theta_*}}\lb T_1 \rb \cdot \EE_{p_{\theta_*}}\lb T_2 \rb = 0.
    \end{equation*}
\end{remark}
\subsection{Examples and numerical experiments of two efficiencies}\label{num-exp}
\begin{example}[Gaussian distribution]
 Consider the Gaussian distribution with mean value $\mu$ and standard variance $\sigma$: 
\begin{equation*}
p(x;\mu,\sigma)=\frac{1}{\sqrt{2\pi} \sigma}e^{-\frac{1}{2\sigma^2}(x-\mu)^2}.
\end{equation*}
The WIM satisfies 
\begin{equation*}
G_W(\mu,\sigma)=\begin{pmatrix}
1 & 0 \\
0 &  1
\end{pmatrix}.
\end{equation*}
The Fisher information matrix satisfies 
\begin{equation*}
G_F(\mu,\sigma) = \begin{pmatrix}
			\frac{1}{\sigma^2} & 0		\\
			0 & \frac{2}{\sigma^2}
		\end{pmatrix}.
\end{equation*}
Further, the matrix $G_FG_W^{-1}$ is given by:
\begin{equation*}
    G_F(\mu,\sigma)G_W^{-1}(\mu,\sigma) = \begin{pmatrix}
			\frac{1}{\sigma^2} & 0		\\
			0 & \frac{2}{\sigma^2}
		\end{pmatrix}.
\end{equation*}
And optimal parameters are given by $\mu_*, \sigma_*$. Thus we have following conclusions on efficiency of the Fisher, Wasserstein natural gradients and Wasserstein natural gradient on Fisher score (maximal likelihood estimator).       \\
\indent
The Wasserstein natural gradient is asymptotically efficient with an asymptotic Wasserstein covariance given by:
\begin{equation*}
    V_t = \frac{1}{t}\begin{pmatrix}
1 & 0 \\
0 &  1
\end{pmatrix} + O\lp \frac{1}{t^2} \rpz.
\end{equation*}
The Fisher natural gradient is asymptotic efficient with an asymptotic classical covariance given by:
\begin{equation*}
    V_t = \frac{1}{t}\begin{pmatrix}
\sigma_*^2 & 0 \\
0 &  \frac{\sigma_*^2}{2}
\end{pmatrix} + O\lp \frac{1}{t^2} \rpz.
\end{equation*}   
An interesting thing here is that the covariance matrix appears in the Wasserstein efficiency is independent of the optimal value. While in Fisher case, the asymptotic behavior depends a lot on the optimal parameter we obtain.       \\
\indent
For the last case, in Gaussian family, two metric tensors $G_F,G_W$ can be simultaneously diagonalized, thus the situation is even simpler. We denote the least significant eigenvalue of $G_F G_W^{-1}$ as $\alpha$:
\begin{equation*}
    \alpha = \frac{1}{\sigma_*^2}.
\end{equation*}
Further more, we have to figure out the term
\begin{equation*}
    \mbE_{p_{\mu_*,\sigma_*}} \lb \nabla_x \lp \nabla_{\mu_*,\sigma_*}l(x_t,\mu_*,\sigma_*) \rpz \cdot \nabla_x \lp \nabla_{\mu_*,\sigma_*}l(x_t,\mu_*,\sigma_*)^{T} \rpz  \rb,
\end{equation*}
that appears in the final result. In Gaussian, since we have Fisher scores $\nabla_{\mu_*,\sigma_*} l(x;\theta) = \Phi^F(x,\mu_*,\sigma_*)$ as:
\begin{equation*}
\begin{split}
\Phi_\mu^F(x;\mu,\sigma) = \frac{x - \mu}{\sigma^2}, \quad \Phi_\sigma^F(x;\mu,\sigma) = \frac{(x-\mu)^2}{\sigma^3} - \frac{1}{\sigma}.
\end{split}
\end{equation*}
Via calculation, we have
\begin{equation*}  
    \begin{aligned}
        \mbE_{p_{\mu_*,\sigma_*}} \lb \nabla_x \Phi_\mu^F(x;\mu_*,\sigma_*) \cdot \nabla_x \lp \Phi_\mu^F(x;\mu_*,\sigma_*)^{T} \rpz  \rb = & \ \mbE_{p_{\mu_*,\sigma_*}}\lb \frac{1}{\sigma_*^4} \rb = \frac{1}{\sigma_*^4},      \\
        \mbE_{p_{\mu_*,\sigma_*}} \lb \nabla_x \Phi_\mu^F(x;\mu_*,\sigma_*) \cdot \nabla_x \lp \Phi_{\sigma}^F(x;\mu_*,\sigma_*)^{T} \rpz  \rb = & \ \mbE_{p_{\mu_*,\sigma_*}}\lb \frac{1}{\sigma_*^2} \cdot \frac{2(x - \mu_*)}{\sigma_*^3} \rb = 0,      \\
        \mbE_{p_{\mu_*,\sigma_*}} \lb \nabla_x \Phi_{\sigma}^F(x;\mu_*,\sigma_*) \cdot \nabla_x \lp \Phi_{\sigma}^F(x;\mu_*,\sigma_*)^{T} \rpz  \rb = & \ \mbE_{p_{\mu_*,\sigma_*}}\lb \frac{4\lp x - \mu_* \rpz^2}{\sigma_*^6} \rb = \frac{4}{\sigma_*^4},
    \end{aligned}
\end{equation*}
we conclude the middle term is given by
\begin{equation*}
    \mathfrak{I} = \mbE_{p_{\mu_*,\sigma_*}} \lb \nabla_x \lp \nabla_{\mu_*,\sigma_*}l(x_t,\mu_*,\sigma_*) \rpz \cdot \nabla_x \lp \nabla_{\mu_*,\sigma_*}l(x_t,\mu_*,\sigma_*)^{T} \rpz  \rb = \begin{pmatrix}
        \frac{1}{\sigma_*^4} & 0        \\
        0 & \frac{4}{\sigma_*^4}
    \end{pmatrix}.
\end{equation*}
And when we have $\frac{2}{\sigma_*^2} > 1$, the inverse matrix of $2B - \I$ is given by
\begin{equation*}
    \lp 2B - \I \rpz^{-1} = \begin{pmatrix}
        \frac{\sigma_*^2}{2 - \sigma_*^2} & 0  \\
        0 & \frac{\sigma_*^2}{4 - \sigma_*^2}
    \end{pmatrix}.
\end{equation*}
Consequently, the term appearing in the asymptotic behavior of the Poincar\'{e} efficiency is given by
\begin{equation*}
    \begin{aligned}
        & \frac{1}{t}\lp 2G_FG_W^{-1} - \I \rpz^{-1}G_W^{-1}(\theta_*) \mathfrak{I} \left( G_W^{-1}(\theta_*) \right)        \\
        = & \begin{pmatrix}
        \frac{\sigma_*^2}{2 - \sigma_*^2} & 0  \\
        0 & \frac{\sigma_*^2}{4 - \sigma_*^2}
        \end{pmatrix}\begin{pmatrix}
        \frac{1}{\sigma_*^4} & 0        \\
        0 & \frac{4}{\sigma_*^4}
    \end{pmatrix}      \\
        = & \begin{pmatrix}
        \frac{1}{\lp 2 - \sigma_*^2 \rpz \sigma_*^2} & 0  \\
        0 & \frac{4}{\lp 4 - \sigma_*^2 \rpz \sigma_*^2}
    \end{pmatrix}.
    \end{aligned}
\end{equation*}
Thus the asymptotic behavior the Wasserstein covariance in the Wasserstein natural gradient of Fisher scores is given by:
\begin{equation*}
       V_t = \left\{
    \begin{aligned}
        & O\lp t^{-\frac{2}{\sigma_*^2}} \rpz, \qquad \qquad \qquad \qquad \qquad \quad  \ \ \frac{1}{\sigma_*^2} \leq \half,        \\
        & \frac{1}{t}\begin{pmatrix}
        \frac{1}{\lp 2 - \sigma_*^2 \rpz \sigma_*^2} & 0  \\
        0 & \frac{4}{\lp 4 - \sigma_*^2 \rpz \sigma_*^2}
    \end{pmatrix} + O(\frac{1}{t^2}), \qquad \frac{1}{\sigma_*^2} > \half.
    \end{aligned}\right. 
\end{equation*}
\par
We verify our theory by following numerical experiments. In two cases, we verify two kinds of efficiency, namely the Wasserstein-Cramer-Rao efficiency and the Poincar\'{e} efficiency respectively. In the first experiment, we verify the constant $G_W^{-1}$ appearing in asymptotic efficiency of the Wasserstein natural gradient. While for the other situation we verify the asymptotic exponential index $\alpha$ showing up in Poincar{\'e} efficiency.
\begin{figure}[H]
  \centering
  \centerline{\includegraphics[width=0.8\linewidth]{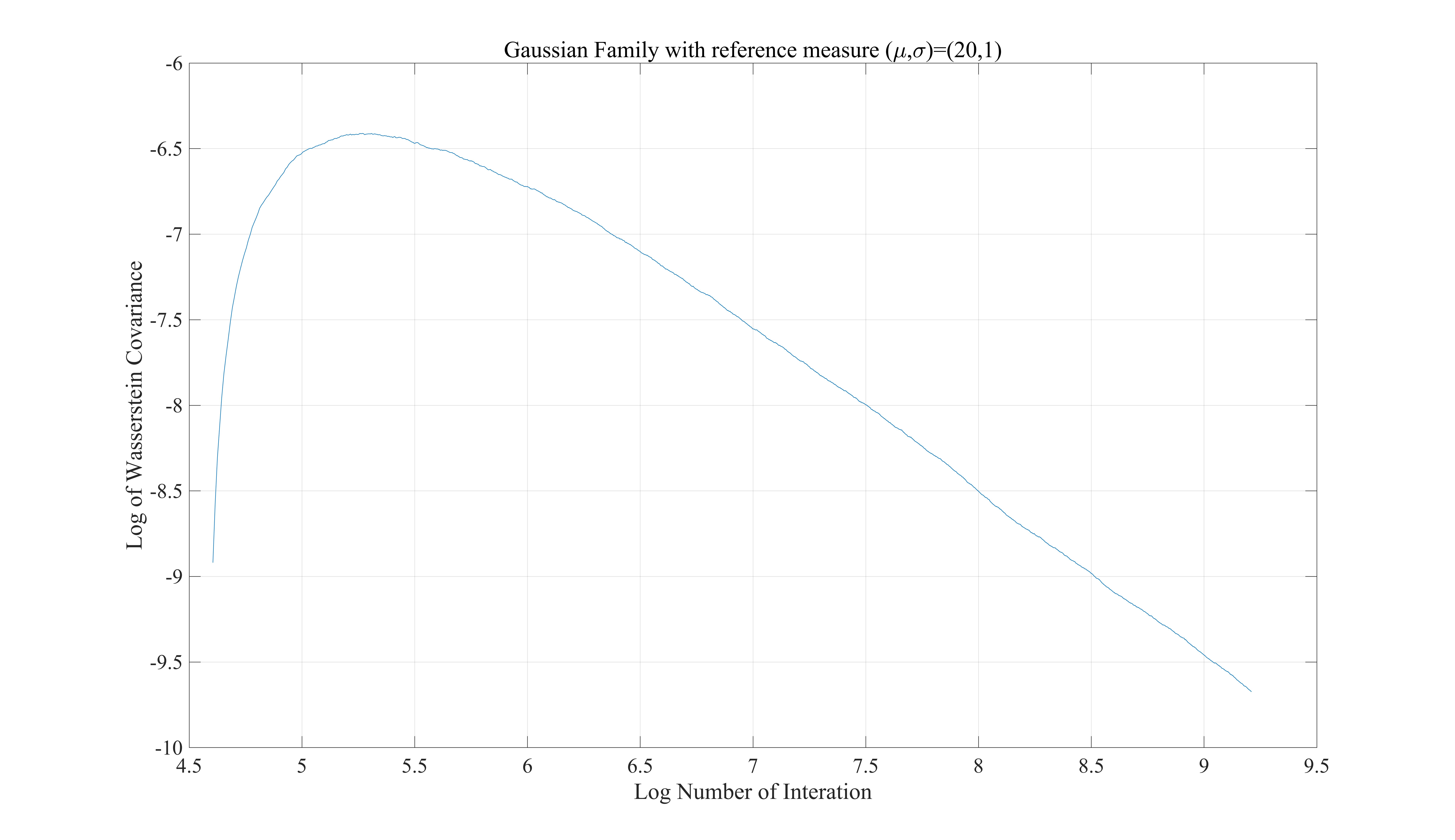}}
  \caption{The Wasserstein-Cramer-Rao Type Convergence Rate. Here x-axis represents the logarithm of iteration $t$ while y-axis represents the logarithm of Wasserstein covariance $V_t$. We take the reference measure in KL-divergence to be Gaussian $\cn\lp 20, 1 \rpz$ where the parameter $\mu_* = 20$ is the optimal point we aim to estimate. Since we have $\frac{1}{\sigma_*^2} = 1 > \half$, the Cramer-Rao type convergence holds.}
  \label{fig:CR}
\end{figure}
\begin{figure}[H]
  \includegraphics[width=0.8\linewidth]{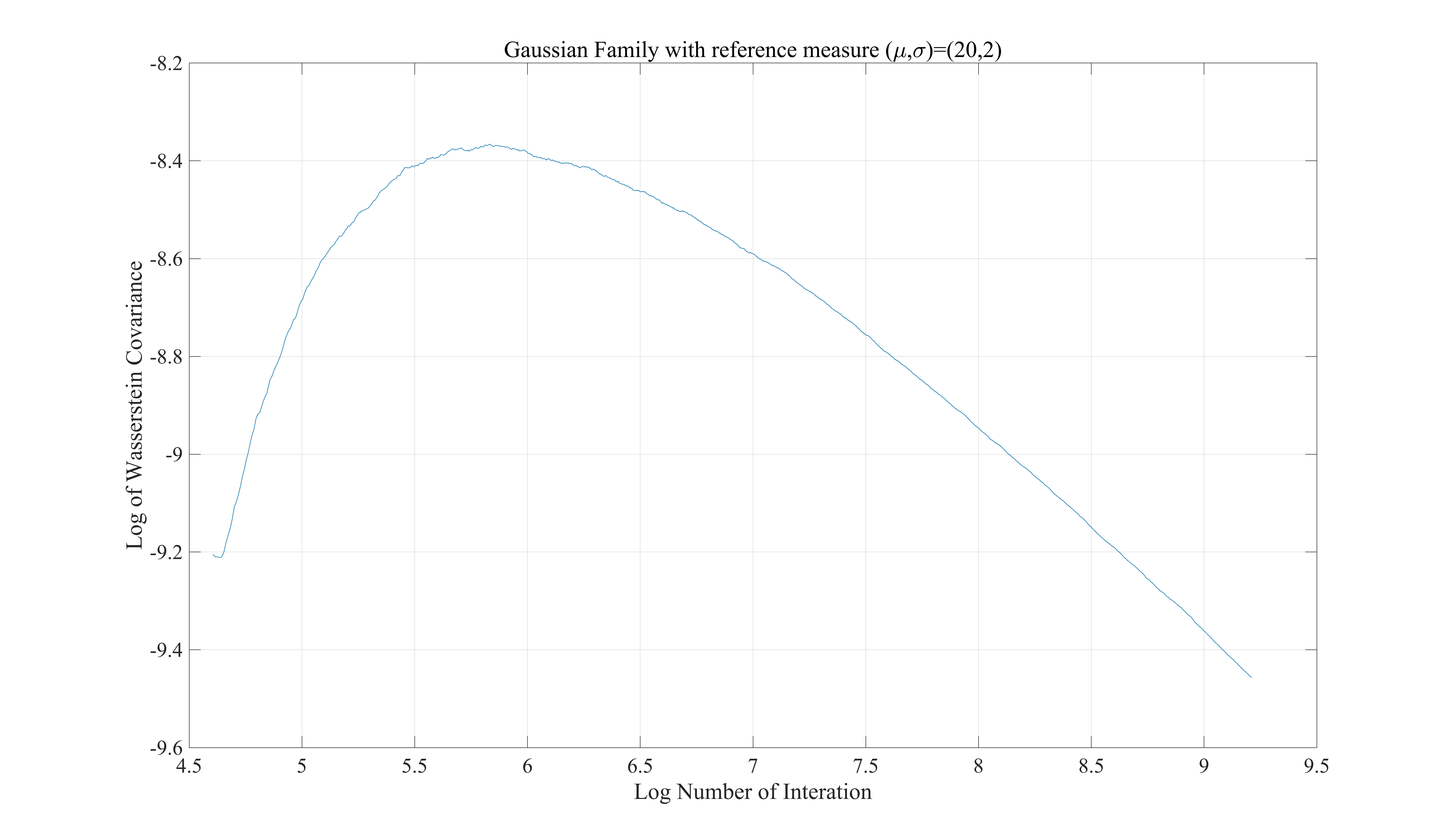}
  \caption{Poincar\'{e} Type Convergence Rate. Here x-axis represents the logarithm of iteration $t$ while y-axis represents the logarithm of Wasserstein covariance $V_t$. We take the reference measure in KL-divergence to be Gaussian $\cn\lp 20, 1 \rpz$ where the parameter $\mu_* = 20$ is the optimal point we aim to estimate. Since we have $\frac{1}{\sigma_*^2} = \frac{1}{4} < \half$, the Poincar\'{e} type convergence holds.}
  \label{fig:Po}
\end{figure}
\end{example}

\end{document}